\documentclass[12pt]{amsart}%

\usepackage{amssymb}
\usepackage{amsmath}
\usepackage{amsthm}
\usepackage{amscd}
\usepackage[margin=1in]{geometry}
\usepackage{hyperref}

\theoremstyle{plain}
\newtheorem{theorem}{Theorem}[section]
\newtheorem{proposition}{Proposition}[section]
\newtheorem{corollary}{Corollary}[section]
\newtheorem{lemma}{Lemma}[section]

\theoremstyle{remark}

\newtheorem{remark}{Remark}[section]
\newtheorem{examples}{Examples}[section]
\newtheorem{assumption}{Assumption}[section]

\DeclareMathOperator*{\esssup}{ess\,sup}

\DeclareMathOperator{\lin}{span}
\DeclareMathOperator{\cpct}{Cap}

\begin{document}

\title[Vector analysis for 
Dirichlet forms]{
Vector analysis for 
Dirichlet forms and quasilinear PDE and SPDE 
on 
metric 
measure 
spaces 
} 


\author[M.Hinz]{Michael Hinz$^1$}
\address{
Department of Mathematics,  
University of Connecticut,
Storrs, CT 06269-3009 USA  and    Fakult\"at f\"ur Mathematik,  
Universit\"at Bielefeld,
Postfach 100131,
D-33501 Bielefeld, Germany}
\email{mhinz@math.uni-bielefeld.de}
\thanks{$^1$Research supported in part by NSF grant DMS-0505622 and by the Alexander von Humboldt Foundation Feodor (Lynen Research Fellowship Program)}
\author[M.R\"ockner]{Michael R\"ockner$^2$}
\address{Fakult\"at f\"ur Mathematik,  
Universit\"at Bielefeld,
Postfach 100131,
D-33501 Bielefeld, Germany}
\email{roeckner@math.uni-bielefeld.de}
\thanks{$^2$Research was supported in part by the German Science Foundation (DFG) through CRC
701}
\author[A.Teplyaev]{Alexander Teplyaev$^3$}
\address{Department of Mathematics, University of Connecticut, Storrs, CT 06269-3009 USA}
\email{Alexander.Teplyaev@uconn.edu}
\thanks{$^3$Research supported in part by NSF grant DMS-0505622}


\date{\today}

\begin{abstract}
Starting with a regular symmetric Dirichlet form on a locally compact separable metric space $X$, 
our paper studies elements of vector analysis, $L_p$-spaces of
vector fields and related Sobolev spaces. 
These tools are then employed to obtain existence and uniqueness results for 
some quasilinear elliptic PDE and  SPDE in variational form on $X$ by standard methods. 
For many of our  results locality is not assumed, but most interesting applications involve local regular Dirichlet forms on fractal spaces such as nested fractals and Sierpinski carpets. 
\tableofcontents
\end{abstract}
\maketitle

\section{Introduction and setup}\label{S:Intro} 

This paper is concerned with some elements of vector analysis on locally compact spaces $X$ that carry a regular Dirichlet form. We start from the notion of $1$-forms based on energy as recently introduced by Cipriani and Sauvageot in \cite{CS03,CS09} and further studied in \cite{IRT}. A priori this concept is of global, non-local nature, and the space $\mathcal{H}$ of $1$-forms defined in \cite{CS03,CS09} is a Hilbert space which in classical smooth cases agrees with the Hilbert space of $L_2$-differential $1$-forms. It is shown below that for local Dirichlet forms this approach may be seen as an extension of closely related and preceding constructions of Eberle, \cite{Eb99}, based on abstract (local) differential operators.
Within the framework of \cite{CS03,CS09} we propose to study some basic notions of vector analysis such as vector fields, gradient and divergence operators. Furthermore, a direct integral representation of $\mathcal{H}$ allows to define $L_p$-spaces over measurable fields of Hilbert spaces (fibers) $(\mathcal{H}_x)_{x\in X}$ such that the space $\mathcal{H}$ of $1$-forms in the sense of \cite{CS03,CS09} appears for $p=2$. Related Sobolev spaces of functions and vector fields come up naturally after that. We show that these tools can be applied to quasilinear elliptic PDE on $X$ in divergence and non-divergence form and  to SPDE on $X$ in variational form such as, for instance, the stochastic $p$-Laplace equation. The proposed notions of vector analysis allows to obtain existence and uniqueness results by classical fixed point and monotonicity arguments. We finally discuss a probabilistic counterpart of $\mathcal{H}$ which goes back to Nakao, \cite{N85}. This allows to give probabilistic interpretations of our vector analysis in terms of additive functionals. The mentioned direct integral representation of $\mathcal{H}$ nicely connects to (analytic and probabilistic) perturbation results for Dirichlet forms, e.g. \cite{FK04}, what permits to define analogs of non-divergence form operators in our context.
 
The main motivation for the present study comes from the analysis on fractals, cf. \cite{Ba,Ki01,Str06}. For certain classes of fractal sets the existence of a Laplace operator has been proved, see \cite{BB99,BBKT,HMT,Ki01,Li,P,Stei10} and the references therein for some examples. Linear elliptic and parabolic PDE on fractals can then be treated by standard methods, \cite{Ev98}. Semilinear equations have been studied in \cite{FaHu01}. There are also methods that apply to fully nonlinear problems, see for instance \cite{BdPR,RRW} for porous medium equations. However, to our knowledge quasilinear equations of type $$div(a(\nabla u))=f$$ or $$\Delta u +b(\nabla u)=f$$
with generally nonlinear coefficients $a$ and $b$ have not been considered so far, 
as an appropriate notion of gradient $\nabla$ on fractals had not yet been sufficiently developed. The present paper addresses these problems. It also establishes a basis for further studies of first order differential operators on fractals, which have never been carried out before. Examples of such operators and related equations include for instance Dirac operators, magnetic Schr\"odinger operators or the Navier-Stokes equations on fractals, investigated in the companion papers \cite{HTb} and \cite{HTa}, respectively. A short survey and other developments can be found in \cite{HTc,HKT,HKT2013}.

In \cite{CS03}, and later in \cite{CS09} and \cite{IRT}, a Hilbert space $\mathcal{H}$ of $1$-forms and a related analog $\partial$ of the exterior derivation (in $L_2$-sense) had been introduced by means of tensor products and energy norms, see Section \ref{S:Energy} below. For the classical Dirichlet form associated with the Laplace-Beltrami operator on a smooth compact Riemannian manifold the space $\mathcal{H}$ agrees with the Hilbert space of $L_2$-differential $1$-forms, up to an isomorphism. The norm in $\mathcal{H}$ is most conveniently expressed in terms of energy measures in the sense of LeJan and Fukushima, \cite{FOT94,LeJan78}. Without too much effort a related notion of weighted energy measures for $1$-forms can be introduced, what yields a coherent picture (especially in probabilistic terms) and is useful for some applications in \cite{HTa} and \cite{HTb}.  

The energy measure of a bounded energy finite function may be absolutely continuous with respect to the given reference measure or not. In Eberle \cite[Section 3.2 and Appendix D]{Eb99}  it is shown how to construct derivation operators if the energy measures are absolutely continuous for all functions from a dense algebra contained in the domain of the generator. On fractal spaces energy measures are typically singular with respect to the self-similar Hausdorff measure on the base space, cf. \cite{BBST99,Hino03,Hino05,Ku89}. However, the construction in \cite[Theorem 3.11]{Eb99}  is still possible if we choose a finite or countable pool of functions admitting energy densities and being energy dense in the space of bounded energy finite functions. Switching to a suitable measure $m$ if needed (a so-called \emph{energy dominant measure} \cite{Hino10} or, more specifically, a \emph{Kusuoka measure} $\widetilde{m}$, see \cite{Ki08,Ku89,T08}), this can be realized for any regular Dirichlet form. 

Following \cite{Eb99} we therefore obtain a measurable field of Hilbert spaces, \cite{Dix,Tak}. Rewriting the construction using some simple manipulations it can be shown that, roughly speaking, the resulting direct integral is a Hilbert space isomorphic to the space of $1$-forms $\mathcal{H}$. Moreover, the direct integral of Eberle's fiberwise operators coincides with the derivation $\partial$ in the sense of Cipriani and Sauvageot in the case of local commutative Dirichlet forms. Apart from minor modifications this material is not new in substance. However, the direct connection between these two constructions had not been well established before. Even more importantly, our reasoning provides a constructive fiberwise interpretation for $\mathcal{H}$ that carries over from \cite{Eb99}. 
Our results imply  that   the construction in \cite{CS03,CS09}  could be viewed as an extension of that in \cite[Theorem 3.11]{Eb99}, now based on a regular Dirichlet form instead of an abstract differential operator.  

By the self-duality of $\mathcal{H}$ we regard its elements also as vector fields and $\partial$ a gradient operator. As a first new result, a corresponding divergence operator is defined as the adjoint of $\partial$. Note that although Eberle considers the adjoint of the derivation operator, \cite[Chapter 3 b), Section 1]{Eb99}, in his case it is part of the basic hypotheses and the discussion there aims at constructing Sobolev spaces of functions rather than at investigating spaces of vector fields.

Under additional assumptions on the given Dirichlet form, its restriction to a suitable core $\mathcal{C}$ is also closable with respect to some energy dominant measure $\widetilde{m}$. This follows using arguments from \cite[Section 6.2]{FOT94} and \cite{Ki12}, respectively.

Using the above mentioned fiberwise interpretation, it is straightforward to define $L_p$-spaces of vector fields. Based on the previous closability result we then introduce Sobolev spaces of functions that make the derivation a closed operator for any $p\geq 2$, provided there exists a core $\mathcal{C}_p$ of functions having $p/2$-integrable energy densities which is dense in $L_p$ and moreover such that $\mathcal{C}_p\otimes\mathcal{C}_p$ is dense in the corresponding $L_p$-space of vector fields. These assumptions are clearly satisfied in the classical smooth context. To verify them for non-classical examples we propose to investigate abstract continuous coordinates with respect to a measure. Harmonic coordinates in the sense of \cite{Ki93h,Ki08,T08} constitute a prototype example. Any symmetric regular Dirichlet form admits such continuous coordinates with respect to the aforementioned energy dominant measure $\widetilde{m}$. Therefore we observe that if the original Dirichlet form is local and transient or induced by a local resistance form, the Sobolev spaces are well defined and the derivations are closed operators. 

The applications to PDE and SPDE follow standard patterns that become applicable thanks to the  definitions and results described above. 

The space $\mathcal{H}$ is isometrically isomorphic to the Hilbert space $\mathring{\mathcal{M}}$ of martingale additive functionals of finite energy as studied for instance in \cite{FOT94, N85}. Given our setup, this isomorphism is almost immediate. Weighted energy measures of $1$-forms and energy measures of martingale additive functionals correspond to each other, and gradients can be understood in terms of the martingale part in the Fukushima decomposition. The divergence may be expressed in terms of Nakao's divergence functional.

Our basic setup is as follows: $X$ is assumed to be a locally compact separable metric space. We do not use the metric explicitely, but need $X$ to be a locally compact and second countable Hausdorff space, and any such space is metrizable. By
$\mathcal{M}(X)$ we denote the space of (signed) Radon measures on $X$ and by $\mathcal{M}_+(X)$ the cone consisting of its non-negative elements; a measure $\mu\in\mathcal{M}_+(X)$ is an \emph{admissible reference measure} on $X$ if each open set $U\subset X$ has positive measure $\mu(U)>0$. In the sequel we assume that $\mu$ is an admissible reference measure on $X$ and, furthermore, we assume that $(\mathcal{E},\mathcal{F})$ is a regular symmetric Dirichlet (energy) form on $L_2(X,\mu)$, cf. \cite{FOT94}. More exactly, we begin our arguments with an admissible reference measure $\mu$, and later switch to an energy dominant measure $m$ if necessary, see Lemma~\ref{L:lem} below. 

Set $\mathcal{C}:=C_0(X)\cap\mathcal{F}$. By regularity the space $\mathcal{C}$ is dense in $\mathcal{F}$. It is an algebra, and if we endow it with the norm $\left\|f\right\|_{\mathcal{C}}:=\mathcal{E}(f)^{1/2}+\sup_X|f|$, we observe
\begin{equation}\label{E:boundmult}
\mathcal{E}(fg)^{1/2}\leq \left\|f\right\|_{\mathcal{C}}\left\|g\right\|_{\mathcal{C}}, \ \ f,g\in\mathcal{C},
\end{equation}
as a consequence of the Markov property, see for instance \cite{BH91}. 
Here we use the notation $\mathcal{E}(f):=\mathcal{E}(f,f)$, and we will do similarly for any other bilinear expression.

For any $g,h\in\mathcal{C}$ there exists a unique signed finite measure $\Gamma(g,h)\in\mathcal{M}(X)$ such that for any $f\in \mathcal{C}$, 
\begin{equation}\label{E:energymeasure}
2\int_X f\:d\Gamma(g,h)=\mathcal{E}(fg,h)+\mathcal{E}(fh,g)-\mathcal{E}(gh,f)\ .
\end{equation}
Obviously $\Gamma:\mathcal{C}\times\mathcal{C}\to\mathcal{M}(X)$ is a well defined symmetric bilinear mapping, and for any $g\in\mathcal{C}$, $\Gamma(g)\in\mathcal{M}_+(X)$. The measure $\Gamma(g)$ is called the \emph{energy measure of $g$}, cf. \cite{FOT94,LeJan78}. Using the estimate 
\begin{equation}\label{E:approxmeasure}
\left|\left(\int_X fd\Gamma(g)\right)^{1/2}-\left(\int_Xfd\Gamma(h)\right)^{1/2}\right|^2\leq 2\sup_{x\in X}|f(x)|\mathcal{E}(g-h),
\end{equation}
\cite[p. 111]{FOT94} we can define (finite) energy measures $\Gamma(g,h)\in\mathcal{M}(X)$ for arbitrary $g,h\in\mathcal{F}$ and even for arbitrary $g,h\in\mathcal{F}_e$, where $(\mathcal{F}_e,\mathcal{E})$ denotes the extended Dirichlet space with respect to $\mu$, that is the collection of $\mu$-measurable $\mu$-a.e. finite functions $g$ on $X$ for which there exists a $\mathcal{E}$-Cauchy sequence $(g_n)_n\subset \mathcal{F}$ such that $\lim_n g_n=g$ $\mu$-a.e.   The form $\mathcal{E}$ extends to $\mathcal{F}_e$ by $\mathcal{E}(g):=\lim_n\mathcal{E}(g_n)$, the limit being independent of the choice of $(g_n)_n$. See \cite{FOT94}. If $\mathcal{E}$ has no killing part, cf. \cite[Theorem 3.2.1]{FOT94}, then 
\begin{equation}\label{E:desirable}
\Gamma(g,h)(X)=\mathcal{E}(g,h), \ \ g,h\in\mathcal{C}.
\end{equation}

\begin{examples}\label{Ex:one}\mbox{}
\begin{enumerate}
\item[(i)] If $\Omega\subset\mathbb{R}^n$ is a smooth bounded domain and $\mathcal{E}(f)=\int_\Omega|\nabla f|_{\mathbb{R}^n}^2 dx$ then $\Gamma(f,g)=\left\langle f,g\right\rangle_{\mathbb{R}^n} dx$.
\item[(ii)] If $M$ is a smooth compact Riemannian manifold and $\mathcal{E}(f)=\int_M |df|_{T^\ast M}^2dvol$, where $d$ denotes the exterior derivative and $dvol$ the volume measure, then $\Gamma(f,g)=\left\langle df, dg\right\rangle_{T^\ast M}dvol$.
\end{enumerate}
\end{examples}

In the next section the definition on the space $\mathcal{H}$ of $1$-forms is given and the concept of energy measure is extended to $1$-forms. A fiberwise perspective is investigated and $\mathcal{H}$ is shown to coincide with the direct integral considered in \cite[Appendix D]{Eb99}. Section \ref{S:gradient} introduces gradient and divergence, equipped with suitable domains, and Section \ref{S:PDE1} presents some applications to quasilinear PDE. In Section \ref{S:closability} we discuss the question of closability when changing from the original to the energy dominant measure. Sobolev spaces and abstract continuous coordinates are introduced in Sections \ref{S:Sobo} and \ref{S:coords}, respectively, while Section \ref{S:PDE} contains some further applications, now to SPDE in the variational framework. We conclude the paper with some remarks on stochastic calculus in Section \ref{S:stochcalc}. To keep notation short, sequences or families indexed by the naturals (or pairs of naturals) will be written with index set suppressed, e.g. $(a_n)_n$ stands for $(a_n)_{n\in\mathbb{N}}$. Similarly, $\lim_n a_n$ abbreviates $\lim_{n\to \infty}a_n$.

\subsection*{Acknowledgements} We are very grateteful to the anonymous referees for their careful reading, their patience and the numerous valuable suggestions they made. In particular, we thank them for communicating the short proof of Lemma \ref{L:zerolevel} as stated. We would also like to thank Dan Kelleher for pointing out some inaccuracies in an earlier version of this paper. The first named author is especially grateful to him for observing that \cite[Lemma 3.1 (i)]{H12} is incorrect. The results of \cite[Lemma 3.1 (ii), Lemma 3.2 and Theorem 5.2]{H12} remain valid and their proofs can be fixed easily.

\section{The space $\mathcal{H}$ and weighted energy measures}\label{S:Energy}

By $\mathcal{B}_b(X)$ we denote the space of bounded Borel functions on $X$. Consider $\mathcal{C}\otimes\mathcal{B}_b(X)$, endowed with the symmetric bilinear form 
\begin{equation}\label{E:scalarprodH}
\left\langle a\otimes b, c\otimes d\right\rangle_{\mathcal{H}}:=\int_X bd\:d\Gamma(a,c),
\end{equation}
$a\otimes b, c\otimes d \in \mathcal{C}\otimes \mathcal{B}_b(X)$, and let $\left\|\cdot\right\|_\mathcal{H}$ denote the associated seminorm on $\mathcal{C}\otimes\mathcal{B}_b(X)$. It is nonnegative definite, see \cite{CS03} or Remark \ref{R:CC} (i) below. We write 
\[ker\:\left\|\cdot\right\|_\mathcal{H}:=\left\lbrace \sum_i a_i\otimes b_i\in\mathcal{C}\otimes\mathcal{B}_b(X): \left\|\sum_i a_i\otimes b_i\right\|_\mathcal{H}=0\right\rbrace \]
(with finite linear combinations) and denote the completion of $\mathcal{C}\otimes\mathcal{B}_b(X)/ker\:\left\|\cdot \right\|_{\mathcal{H}}$ with respect to $\left\|\cdot\right\|_\mathcal{H}$ by $\mathcal{H}$. The space $\mathcal{H}$ is a Hilbert space, and following \cite{CS03, CS09} we refer to $\mathcal{H}$ as the \emph{space of differential $1$-forms on $X$}. Unlike for later constructions we agree to use the same notation $a\otimes b$ for a simple tensor from $\mathcal{C}\otimes\mathcal{B}_b(X)$ and for its equivalence class in $\mathcal{H}$.

\begin{remark}\label{R:CC}\mbox{}
\begin{enumerate}
\item[(i)]  It is not difficult to see that if $a_1,...,a_n\in\mathcal{C}$ and the functions $b_1,...,b_n$ are finite linear combinations of indicator functions associated to a partition of $X$, we have
\[\left\langle \sum_i a_i\otimes b_i, \sum_i a_i\otimes b_i \right\rangle_\mathcal{H}=\sum_i\sum_j \int_X b_ib_j\:d\Gamma(a_i,a_j)\geq 0.\]
By pointwise approximation this nonnegativity is seen to hold for general elements $\sum_i a_i\otimes b_i\in \mathcal{C}\otimes\mathcal{B}_b(X)$.   
\item[(ii)] The space $\mathcal{C}\otimes\mathcal{C}$ is dense in  $\mathcal{H}$, and therefore  $\mathcal{H}$ can be constructed from $\mathcal{C}\otimes\mathcal{C}$ in an analogous manner. 
\end{enumerate}
\end{remark}

\begin{examples}\label{Ex:two}
For the Dirichlet form as in Examples \ref{Ex:one} (ii) we obtain
\[\left\langle g_1\otimes f_1, g_2\otimes f_2\right\rangle_{\mathcal{H}}=\int_M f_1f_2\left\langle dg_1,dg_2\right\rangle_{T^\ast M}dvol,  \ \ f_1,f_2,g_1,g_2\in C^\infty(M).\]
Below we will see that up to an isomorphism, $\mathcal{H}$ coincides with the Hilbert space $L_2(M,dvol, T^\ast M)$ of $L_2$-differential $1$-forms on $M$.
\end{examples}

The space $\mathcal{H}$ becomes a bimodule if we declare the algebras $\mathcal{C}$ and $\mathcal{B}_b(X)$ to act on it in the following manner: For $a\otimes b \in  \mathcal{C}\otimes \mathcal{B}_b(X)$, $c\in \mathcal{C}$ and $d\in\mathcal{B}_b(X)$ set
\begin{equation}\label{E:left}
c(a\otimes b):=(ca)\otimes b - c\otimes (ab)\ 
\end{equation}
and 
\begin{equation}\label{E:right}
(a\otimes b)d:=a\otimes (bd).
\end{equation} 
As shown in \cite{CS03} and \cite{IRT}, (\ref{E:left}) and (\ref{E:right}) extend to well defined left and right actions of the algebras
$\mathcal{C}$ and $\mathcal{B}_b(X)$, respectively. In particular, we have 
\[\left\|c(a\otimes b)\right\|_\mathcal{H} \leq \sup_X|c|\left\|a\otimes b\right\|_\mathcal{H} \ \ \text{ and }\ \  \left\|(a\otimes b)d\right\|_\mathcal{H}\leq \sup_X|d|\left\|a\otimes b\right\|_\mathcal{H}. \] 
If $(\mathcal{E},\mathcal{F})$ is local, then the Leibniz rule for energy measures \cite[Lemma 3.2.5]{FOT94} together with (\ref{E:scalarprodH})
implies that left and right multiplication agree for any $c\in\mathcal{C}$, and by approximation they are seen to agree for all $c\in\mathcal{B}_b(X)$. See \cite{H12} or \cite{IRT} for further details. 

We continue the preceding ideas and develop a \emph{global perspective}. The following results apply even if the energy measures are possibly not absolutely continuous with respect to the reference measure $\mu$.
From $\Gamma$ an $\mathcal{M}(X)$-valued bilinear mapping on $\mathcal{H}$ can be constructed. For simple tensors $a\otimes b, c\otimes d\in\mathcal{H}$ set 
\begin{equation}\label{E:energyh}
\Gamma_{\mathcal{H}}(a\otimes b, c\otimes d):=bd\:\Gamma(a,c)\ ,
\end{equation}
seen as an $\mathcal{M}(X)$-equality.

\begin{lemma}\label{L:nonnegmeasure}
(\ref{E:energyh}) extends to a well defined and uniquely determined symmetric bilinear mapping $\Gamma_{\mathcal{H}}:\mathcal{H}\times\mathcal{H}\to\mathcal{M}(X)$ such that for any $\omega\in\mathcal{H}$, $\Gamma_\mathcal{H}(\omega)\in\mathcal{M}_+(X)$. For any $\omega,\eta\in\mathcal{H}$ we have $\Gamma_{\mathcal{H}}(\omega,\eta)(X)=\left\langle \omega,\eta\right\rangle_\mathcal{H}$. 
\end{lemma}

The measure $\Gamma_{\mathcal{H}}(\omega,\eta)$ may be interpreted as a \emph{weighted energy measure}.

\begin{examples}
In Examples \ref{Ex:two} we have $\Gamma_\mathcal{H}(g_1\otimes f_1, g_2\otimes f_2)=f_1f_2\left\langle dg_1,dg_2\right\rangle_{T^\ast M}dvol$. 
\end{examples}

\begin{proof}
First note that for any finite linear combination $\sum_i a_i\otimes b_i\in\mathcal{C}\otimes\mathcal{B}_b(X)$ and any $\varphi\in \mathcal{B}_b(X)$ with $\varphi\geq 0$ we have
\[\sum_i\sum_j\int_X\varphi b_ib_j\:d\Gamma(a_i,a_j)=\left\|\sum_i a_i\otimes (\sqrt{\varphi}b_i)\right\|^2_{\mathcal{H}}\geq 0\]
by definition (\ref{E:right}), hence $\sum_i\sum_j b_ib_j\:d\Gamma(a_i,a_j)$ is a member of $\mathcal{M}_+(X)$. Therefore, if $\sum_i a_i\otimes b_i$ has zero norm, $\sum_i\sum_j b_ib_j\:d\Gamma(a_i,a_j)$ is the zero measure. Now consider finite linear combinations $\sum_i f_i\otimes g_i\in\mathcal{H}$. For each $i$ let $\widetilde{f_i}\otimes \widetilde{g_i}\in\mathcal{C}\otimes\mathcal{B}_b(X)$ be a representant of $f_i\otimes g_i\in \mathcal{H}$ and set
\begin{equation}\label{E:GammaH}
\Gamma_{\mathcal{H}}\left(\sum_i f_i\otimes g_i\right):=\sum_i\sum_j \widetilde{g_i}\widetilde{g_j}\Gamma(\widetilde{f_i},\widetilde{f_j}).
\end{equation} 
By the previous arguments (\ref{E:GammaH}) is a well defined element of $\mathcal{M}_+(X)$. Given a general $1$-form $\omega\in\mathcal{H}$, let $(\omega_k)_k$ be a sequence of finite linear combinations $\omega_k=\sum_{i=1}^{n_k} f_i^{(k)}\otimes g_i^{(k)}\in\mathcal{H}$
approximating $\omega$ in $\mathcal{H}$. For a non-negative function $\varphi\in\mathcal{B}_b(X)$ obviously $\sqrt{\varphi}\in\mathcal{B}_b(X)$ and by (\ref{E:right}),
\[\lim_k\int_X\varphi\:d\Gamma_{\mathcal{H}}(\omega_k)=\lim_k\sum_i\sum_j\int_X\varphi\: \widetilde{g_i}^{(k)}\widetilde{g_j}^{(k)}d\Gamma(\widetilde{f_i}^{(k)},\widetilde{f_j}^{(k)})
=\lim_k\left\|\omega_k\sqrt{\varphi}\right\|_{\mathcal{H}}^2
=\left\|\omega\sqrt{\varphi}\right\|_{\mathcal{H}}^2.\]
Set 
\begin{equation}\label{E:positive}
\Gamma_{\mathcal{H}}(\omega)(\varphi):=\lim_k\int_X\varphi\:d\Gamma_{\mathcal{H}}(\omega_k).
\end{equation}
For arbitrary $\varphi\in\mathcal{B}_b(X)$ consider the standard decomposition $\varphi=\varphi_+-\varphi_-$ with
$\varphi_+=\max(\varphi, 0)$, $\varphi_-=\max(-\varphi,0)$ and define a linear functional on $\mathcal{B}_b(X)$ by
\begin{equation}\label{E:defaslimit}
\Gamma_{\mathcal{H}}(\omega)(\varphi):=\lim_k\int_X\varphi\:d\Gamma_{\mathcal{H}}(\omega_k)=\lim_k\int_X\varphi_+\:d\Gamma_{\mathcal{H}}(\omega_k)-\lim_k\int_X\varphi_-\:d\Gamma_{\mathcal{H}}(\omega_k).
\end{equation}
As this equals $\left\|\omega\sqrt{\varphi_+}\right\|_{\mathcal{H}}^2-\left\|\omega\sqrt{\varphi_-}\right\|_{\mathcal{H}}^2$,
we have
\begin{equation}\label{E:Gammabound}
|\Gamma_{\mathcal{H}}(\omega)(\varphi)|\leq 2\sup_x|\varphi(x)|\left\|\omega\right\|_{\mathcal{H}}^2.
\end{equation}
(\ref{E:defaslimit}) and (\ref{E:Gammabound}) hold in particular for any $\varphi\in C_0(X)$, (\ref{E:positive}) is non-negative if $\varphi\geq 0$. Hence by the Riesz representation theorem there exists a unique non-negative Radon measure $\Gamma_{\mathcal{H}}(\omega)\in\mathcal{M}_+(X)$ such that $\int_X\varphi\:d\Gamma_{\mathcal{H}}(\omega)=\Gamma_{\mathcal{H}}(\omega)(\varphi)$ for all $\varphi\in C_0(X)$.
By (\ref{E:Gammabound}) and denseness this extends to all $\varphi\in C_b(X)$, and $\Gamma_{\mathcal{H}}(\omega)$ is seen to be the weak limit of the measures $\Gamma_{\mathcal{H}}(\omega_k)$.  
Finally, a corresponding bilinear mapping $\Gamma_{\mathcal{H}}$ can be defined via polarization, and the last statement of the lemma follows easily  from (\ref{E:GammaH}) and (\ref{E:positive}).
\end{proof}

To the support of the measure $\Gamma_{\mathcal{H}}(\omega)$ we refer as the \emph{support of the $1$-form $\omega\in\mathcal{H}$}.

\begin{corollary}\label{C:measure}\mbox{} 
\begin{enumerate}
\item[(i)] If $\omega\in \mathcal{H}$ is such that $\left\|\omega\right\|_{\mathcal{H}}=0$,
then $\Gamma_{\mathcal{H}}(\omega)=0$ in $\mathcal{M}(X)$. 
\item[(ii)] For any $\omega,\eta\in\mathcal{H}$ and any Borel set $A$, $|\Gamma_{\mathcal{H}}(\omega,\eta)|(A)\leq \Gamma_{\mathcal{H}}(\omega)(A)^{1/2}\Gamma_{\mathcal{H}}(\eta)(A)^{1/2}$ for any Borel set $A\in\mathcal{B}(X)$. In particular, $\Gamma_{\mathcal{H}}(\omega,\eta)=0$ in $\mathcal{M}(X)$ if $\omega$ and $\eta$ have disjoint supports. 
\end{enumerate}
\end{corollary}

\begin{proof}
(i) is a consequence of (\ref{E:Gammabound}). The first statement in (ii) follows by a standard argument, see e.g. \cite[Proposition 3.3]{M80}: By Lemma \ref{L:nonnegmeasure},
\[0\leq \Gamma_{\mathcal{H}}(\omega-\lambda\eta)=\Gamma_{\mathcal{H}}(\omega)-2\lambda\Gamma_{\mathcal{H}}(\omega,\eta)+\lambda^ 2\Gamma_{\mathcal{H}}(\eta).\]
For any relatively compact Borel set $A$ and any $\lambda>0$,
\[|\Gamma_{\mathcal{H}}(\omega,\eta)|(A)\leq \frac{1}{2}\left(\lambda\Gamma_{\mathcal{H}}(\omega)(A)+\lambda^{-1}\Gamma_{\mathcal{H}}(\eta)(A)\right).\]
If, without loss of generality, $\Gamma_{\mathcal{H}}(\eta)=0$, then we can let $\lambda$ go to zero to see the left hand side is zero. If both $\Gamma_{\mathcal{H}}(\omega)$ and $\Gamma_{\mathcal{H}}(\eta)$ are nonzero, consider
\[\lambda=\frac{\Gamma_{\mathcal{H}}(\omega)(A)^{1/2}}{\Gamma_{\mathcal{H}}(\eta)(A)^{1/2}}\]
to arrive at the desired inequality. By the regularity properties of the measures it extends to arbitrary Borel sets. The last statement in (ii) is a simple consequence.
\end{proof}

\begin{remark}
In the present paper the weighted energy measures $\Gamma_\mathcal{H}(\omega,\eta)$ will not play a predominant role. However, they are substantially used in \cite{HTa} and \cite{HTb}, and we feel that for systematic reasons they should be discussed here.
\end{remark}

The above picture can be complemented by a \emph{fiberwise perspective}. The following fact is well known, see for instance \cite[Lemmas 2.2-2.4]{Hino10}. For the convenience of the reader we briefly sketch it.
\begin{lemma}\label{L:lem}
Given a 
regular Dirichlet form $(\mathcal{E},\mathcal{F})$ on $L_2(X,\mu)$, it is always possible to construct an admissible reference measure $\widetilde{m}$ such that for all $f\in\mathcal{C}$, the measure $\Gamma(f)$ is absolutely continuous with respect to $\widetilde{m}$ and the density $\frac{d\Gamma(f)}{d\widetilde{m}}$ is in $L_1(X,\widetilde{m})$. Moreover, $\widetilde{m}$ may chosen to be finite.
\end{lemma}

As usual we write $\mathcal{E}_1(f):=\mathcal{E}(f)+\left\|f\right\|_{L_2(X,\mu)}^2$, $f\in\mathcal{F}$.

\begin{proof} As $(\mathcal{F},\mathcal{E}_1)$ is a separable Hilbert space, it possesses a countable dense subset $\left\lbrace e_n\right\rbrace_n$ (in practice we may for instance take a countable orthonormal basis and its finite linear combinations with rational coeffcients). For fixed $n$, let $(\varphi_{n,k})_k$ be a sequence of functions from $\mathcal{C}$ such that 
\[\mathcal{E}_1(e_n-\varphi_{n,k})^{1/2}\leq 2^{-k}, \ \ k\in\mathbb{N}.\]
Then $\left\lbrace \varphi_{n,k}\right\rbrace_{k,n}$ is a countable family of functions from $\mathcal{C}$ and also dense in $\mathcal{F}$ with respect to $\mathcal{E}_1$. Let $\left\lbrace \psi_n\right\rbrace_n$ be an enumeration of this family. We may assume that each $\psi_n$ has positive energy. Set
\begin{equation}\label{E:fn}
f_n:=\frac{\psi_n}{\Gamma(\psi_n)(X)^{1/2}}.
\end{equation} 
For each $n\in\mathbb{N}$, $\Gamma(f_n)$ is a probability measure. Let $(U_n)_n$ be an exhaustion of $X$ by a sequence of non-empty relatively compact open sets $U_n\subset X$ with $\overline{U_n}\subset U_{n+1}$, $n\in\mathbb{N}$. Since $\mu$ is an admissible reference measure, we have
$\mu(U_{n+1}\setminus U_n)\geq \mu(U_{n+1}\setminus \overline{U_n})>0$. Now put
\begin{equation}\label{E:mtilde}
\widetilde{m}:=\sum_{n=0}^\infty 2^{-n} \Gamma(f_n) +  \sum_{k=0}^\infty 2^{-k}\mu(U_{k+1}\setminus U_k)^{-1}\mu|_{U_{k+1}\setminus U_k}+\mu(U_0)^{-1}\mu|_{U_0}.
\end{equation}
The series obviously converge set-wise, and proceeding as in the proof of Lemma \ref{L:nonnegmeasure} they are also seen to converge in the weak topology. For any $f\in\mathcal{C}$ there is some approximating sequence $(f_{n_j})_j$ and by construction each $\Gamma(f_{n_j})$ is absolutely continuous with respect to $\widetilde{m}$. If $B\in \mathcal{B}(X)$ is such that $\widetilde{m}(B)=0$, then $\Gamma(f_{n_j})(B)=0$ for all $j$ and since
\[|\Gamma(f)(B)^{1/2}-\Gamma(f_{n_j})(B)^{1/2}|^2\leq 2 \mathcal{E}(f-f_{n_j}),\]
by (\ref{E:approxmeasure}), we have $\Gamma(f)(B)=0$, too. Since $\mu(B)>0$ implies $\widetilde{m}(B)>0$ for any $B\in\mathcal{B}(X)$, the measure $\widetilde{m}$ is an admissible reference measure.
\end{proof}

Let us return to the fixed regular symmetric Dirichlet form $(\mathcal{E},\mathcal{F})$ on $L_2(X,\mu)$ as used in (\ref{E:energymeasure}) and (\ref{E:scalarprodH}). From now on we assume the following:  
\begin{assumption}\label{A:energydominant}
\text{The measure $m$ is an admissible reference measure such that for any $f\in\mathcal{C}$,} \\
\text{the measure $\Gamma(f)$ is absolutely continuous with respect to $m$.}
\end{assumption} 
Note that in this case $\Gamma(f)=\frac{d\Gamma(f)}{dm}$ is in $L_1(X,m)$ for any $f\in\mathcal{C}$. If all energy measures $\Gamma(f)$, $f\in\mathcal{C}$, are absolutely continuous with respect to $\mu$, we may use $m:=\mu$. If not, we switch to the measure $m:=\widetilde{m}$ from Lemma \ref{L:lem}. As this is sufficient for later purposes, the above assumption is no additional restriction.

\begin{remark} If $(\mathcal{E},\mathcal{F})$ is irreducible or transient or if it is induced by a regular resistance form then $(\mathcal{E},\mathcal{C})$ can be shown to be closable in $L_2(X,\widetilde{m})$. This will be discussed in Section \ref{S:Sobo}. In the present and the next two sections closability is not needed.
\end{remark}

We recall a construction from \cite{Eb99}. Let $\mathcal{A}_0=\left\lbrace f_n\right\rbrace_n$ be a countable collection of functions which is $\mathcal{E}$-dense in $\mathcal{C}$, i.e. such that for any $f\in\mathcal{C}$ there exists a sequence $(f_{n_j})_j\subset \mathcal{A}_0$ with  $\lim_j\mathcal{E}(f-f_{n_j})=0$. For any finite linear combination $u=\sum_{i=1}^N\lambda_if_i$ and any Borel set $A\subset X$ we have 
\[0\leq \Gamma(u)(A)=\sum_{i=1}^N\sum_{j=1}^N\lambda_i\lambda_j\int_A\Gamma_x(f_i,f_j)m(dx)=\int_A \overline{\lambda}^T(\Gamma_x(f_i,f_j))_{i,j=1,...,N}\overline{\lambda}\:m(dx) ,\]
where $\overline{\lambda}=(\lambda_1,...,\lambda_N)\in\mathbb{R}^N$ and $\overline{\lambda}^T$ is its transpose. We can therefore choose Borel versions $x\mapsto \Gamma_x(f_i,f_j)$ of the classes $\Gamma(f_i,f_j)\in L_1(X,m)$ such that for all $N\in\mathbb{N}$ and all $x\in X$, the matrix $(\Gamma_x(f_i,f_j))_{i,j=1,...,N}$ is symmetric and nonnegative definite over $\mathbb{Q}^N$. For two finite linear combinations $u=\sum_i\lambda_if_i$ and $v=\sum_j\mu_jf_j$ from $\mathcal{A}:=\lin(\mathcal{A}_0)$ set
\[\Gamma_x(u,v):=\sum_i\sum_j\lambda_i\mu_j\Gamma_x(f_i,f_j) .\]
Then for all $x\in X$, $\Gamma_{x}$ clearly is a non-negative definite bilinear form on $\mathcal{A}$. Consider the factor $\mathcal{A}/ker\:\Gamma_x$, where  $ker\:\Gamma_x:=\left\lbrace f\in\mathcal{A}: \Gamma_x(f)=0\right\rbrace$ and let $d_xf$ denote the equivalence class of $f\in\mathcal{A}$. Define
\begin{equation}\label{E:EberleAnsatz}
(d_xf,d_xg)_{\mathcal{B}_x}=\Gamma_x(f,g)
\end{equation}
for all $f,g\in\mathcal{A}$ and let $\mathcal{B}_x$ denote the completion of $\mathcal{A}/ker\:\Gamma_x$ in $(\cdot,\cdot)_{\mathcal{B}_x}$, clearly a Hilbert space.  

For convenience we recall the following definitions: A collection $(H_x)_{x\in X}$ of Hilbert spaces $(H_x, (\cdot,\cdot)_{H_x})$ together with a subspace $M$ of $\prod_{x\in X} H_x$ is called a \emph{measurable field of Hilbert spaces} if 
\begin{enumerate}
\item[(i)] an element $\xi\in \prod_{x\in X}H_x$ is in $M$ if and only if $x\mapsto (\xi,\eta)_{H_x}$ is measurable for any $\eta\in M$,
\item[(ii)] there exists a countable set $\left\lbrace \xi_i: i\in\mathbb{N}\right\rbrace\subset M$ such that for all $x\in X$ the span of $\left\lbrace \xi_i(x):i\in\mathbb{N}\right\rbrace$ is dense in $H_x$.
\end{enumerate}
The elements of $M$ are usually referred to as \emph{measurable sections}. Two measurable fields of Hilbert spaces $(H_x)_{x\in X}$ and $(\widetilde{H}_x)_{x\in X}$ are \emph{essentially isometric} if there are a null set $N\subset X$ and a collection $(\Phi_x)_{x\in X\setminus\mathcal{N}}$ of isometries $\Phi_x:H_x\to\widetilde{H}_x$ such that $\xi\in \prod_{x\in X} H_x$ is a member of $M$ if and only if $x\mapsto \Phi_x(\xi(x))\in\widetilde{M}$. If $N$ may be chosen to be empty, we say that $(H_x)_{x\in X}$ and $(\widetilde{H}_x)_{x\in X}$ are \emph{isometric}.

\begin{remark}\label{R:ortho}
Orthonormalizing the $\xi_i$ from (ii) in the respective spaces one obtains the following useful fact: There is  a countable set $\left\lbrace \eta_i: i\in\mathbb{N}\right\rbrace\subset M$ such that for any $x$ with $H_x$ infinite-dimensional, it provides a orthonormal basis and for any $x$ with $dim\:H_x=d(x)$, $\eta_1(x),...,\eta_{d(x)}(x)$ is an orthonormal basis and $\eta_i(x)=0$, $i>d(x)$. For a proof see \cite[Proposition II.4.1]{Dix} or \cite[Lemma 8.12]{Tak}. Note that every $\eta_i(x)$ is a finite linear combination of elements $\xi_j(x)$. $\left\lbrace \eta_i: i\in\mathbb{N}\right\rbrace\subset M$ is then referred to as a \emph{measurable field of orthogonal bases}.
\end{remark}

\begin{lemma}\label{L:essentially}\mbox{}
\begin{enumerate}
\item[(i)] The collection $(\mathcal{B}_x)_{x\in X}$ is a measurable field of Hilbert spaces.
\item[(ii)] Different choices of versions above lead to essentially isometric fields of Hilbert spaces.
\end{enumerate}
\end{lemma}
\begin{proof}
Let $\mathcal{M}$ be the subspace of all $\xi\in \prod_{x\in X} \mathcal{B}_x$ such that $x\mapsto (\xi(x),d_xf_n)_{\mathcal{B}_x}$ is measurable for any $n$. Obviously all $x\mapsto d_xf$, $f\in\mathcal{A}$, are in $\mathcal{M}$. For general $\xi\in \prod_{x\in X} \mathcal{B}_x$ and each $x\in X$ there is a sequence $(g_k)\subset\mathcal{A}$ such that
\[\lim_k\left\|\xi(x)-d_xg_k\right\|_{\mathcal{B}_x}=0.\]
Hence a section $\xi$ is in $\mathcal{M}$ if and only if $x\mapsto (\xi(x),d_xf_n)_{\mathcal{B}_x}$ are measurable for all $n\in\mathbb{N}$. This shows (i).

To see (ii), assume $x\mapsto \widetilde{\Gamma}_x(f_i,f_j)$ are further versions of $\Gamma(f_i,f_j)\in L_1(X,m)$ so that the previous agreements are valid and denote the similarly constructed spaces by $\widetilde{\mathcal{B}}_x$. Then there exists a null set $\mathcal{N}$ such that 
\[(\widetilde{d}_xf_i,\widetilde{d}_xf_j)_{\widetilde{\mathcal{B}}_x}=(d_xf_i,d_xf_j)_{\mathcal{B}_x}\]
for all $i,j\in\mathbb{N}$ and $x\in X\setminus\mathcal{N}$. By the denseness of $\mathcal{A}/ker\:\Gamma_x$ in $\mathcal{B}_x$ and $\mathcal{A}/ker\:\widetilde{\Gamma_x}$ in $\widetilde{\mathcal{B}}_x$ we obtain a unique isometry $\Phi_x$ from $\mathcal{B}_x$ onto $\widetilde{\mathcal{B}}_x$ for any $x\in X\setminus\mathcal{N}$. If now $\xi\in\mathcal{M}$ then
\[(\Phi_x(\xi(x)),\widetilde{d}_xf_n)_{\widetilde{\mathcal{B}}_x}=(\xi(x),d_xf_n)_{\mathcal{B}_x} \]
for $x\in X\setminus\mathcal{N}$ and all $n\in\mathbb{N}$, and the right-hand side is a measurable function of $x$. Therefore $\Phi_x(\xi(x))$ is a measurable section. Similarly for the converse direction.
\end{proof}

This construction may be rephrased as follows. For any point $x\in X$ and arbitrary simple tensors $a\otimes b, c\otimes d\in\mathcal{A}\otimes\mathcal{B}_b(X)$ put
\begin{equation}\label{E:scalarfiber}
\Gamma_{\mathcal{H},x}(a\otimes b, c\otimes d):=b(x)d(x)\Gamma_x(a,c).
\end{equation}
As a consequence of the above choice of versions every $\Gamma_{\mathcal{H},x}$, $x\in X$, defines a non-negative definite bilinear form on $\mathcal{A}\otimes\mathcal{B}_b(X)$. Set $$ker\:\Gamma_{\mathcal{H},x}:=\left\lbrace \sum_i a_i\otimes b_i\in\mathcal{A}\otimes\mathcal{B}_b(X): \Gamma_{\mathcal{H},x}(\sum_i a_i\otimes b_i)=0\right\rbrace$$ and let $\mathcal{H}_x$ be the Hilbert space obtained as the completion of $\mathcal{A}\otimes\mathcal{B}_b(X)/ker\:\Gamma_{\mathcal{H},x}$ with respect to scalar product determined by
\[\left\langle[a\otimes b]_x, [c\otimes d]_x\right\rangle_{\mathcal{H}_x}=\Gamma_{\mathcal{H},x}(a\otimes b, c\otimes d),\]
where $[a\otimes b]_x\in \mathcal{A}\otimes\mathcal{B}_b(X)/ker\:\Gamma_{\mathcal{H},x}$ denotes the equivalence class of $a\otimes b$. 
Note that 
\begin{equation}\label{E:pullout}
[a\otimes b]_x=[a\otimes b(x)]_x=b(x)[a\otimes \mathbf{1}]_x \text{ for any $x\in X$} 
\end{equation}
and any $a\otimes b\in\mathcal{A}\otimes\mathcal{B}_b(X)$, because $\Gamma_{\mathcal{H},x}(a\otimes(b-b(x))=0$ by (\ref{E:scalarfiber}).

\begin{examples}
For the classical Dirichlet form on a smooth compact Riemannian manifold as in Examples \ref{Ex:one}(ii) and \ref{Ex:two} we have
$\Gamma_{\mathcal{H},x}(g_1\otimes f_1, g_2\otimes f_2)=f_1(x)f_2(x)\left\langle dg_1(x), dg_2(x)\right\rangle_{T_x^\ast M}$ and $\mathcal{H}_x$ is the cotangent space $T_x^\ast M$ at $x\in M$.
\end{examples}

\begin{lemma}\label{L:equal} 
The collection $(\mathcal{H}_x)_{x\in X}$ is a measurable field of Hilbert spaces on $X$. The measurable fields $(\mathcal{H}_x)_{x\in X}$ and $(\mathcal{B}_x)_{x\in X}$ are isometric.
\end{lemma}

\begin{proof}
The first assertion may be seen as in the previous lemma. For any $x\in X$ define a bilinear mapping $\Psi_x:\mathcal{A}/ker\:\Gamma_x\to\mathcal{H}_x$ by
\begin{equation}\label{E:EberleIso}
\Psi_x(d_xa):=[a\otimes\mathbf{1}]_x,\ \ a\in\mathcal{A}.
\end{equation}
Since
\begin{equation}\label{E:isometric}
\left\|\Psi_x(d_xa)\right\|_{\mathcal{H}_x}^2=\left\|[a\otimes\mathbf{1}]_x\right\|_{\mathcal{H}_x}^2=\Gamma_{\mathcal{H},x}(a\otimes\mathbf{1})=\Gamma_x(a)=\left\|d_xa\right\|_{\mathcal{B}_x}^2
\end{equation}
and $d_x\widetilde{a}=d_xa$ if and only if $\Gamma_x(\widetilde{a}-a)=0$, $\Psi_x$ is well defined. By (\ref{E:isometric}) and denseness it extends to a uniquely determined isometry from $\mathcal{B}_x$ into $\mathcal{H}_x$. $\Psi_x$ is also surjective: For any $[a\otimes b]_x\in \mathcal{A}\otimes\mathcal{B}_b(X)/ker\:\Gamma_{\mathcal{H},x}$ consider $b(x)d_xa$. Then by linearity and (\ref{E:pullout}), $\Psi_x(b(x)d_xa)=b(x)[a\otimes\mathbf{1}]_x=[a\otimes b]_x$. On the other hand, $\mathcal{A}\otimes\mathcal{B}_b(X)/ker\:\Gamma_{\mathcal{H},x}$ is dense in $\mathcal{H}_x$.
\end{proof}

\begin{lemma}
The space $\mathcal{A}\otimes\mathcal{B}_b(X)$ is dense in $\mathcal{H}$.
\end{lemma}
\begin{proof}
By construction, any simple tensor $a\otimes b\in \mathcal{C}\otimes\mathcal{B}_b(X)$ can be approximated by elements of $\mathcal{A}\otimes\mathcal{B}_b(X)$.
\end{proof}

Recall that given a measurable field of Hilbert spaces $(H_x)_{x\in X}$,  a measurable section $\xi$ is called \emph{square-integrable} if 
\begin{equation}\label{E:squareint}
\int_X\left\|\xi(x)\right\|_{H_x}^2 m(dx)<\infty.
\end{equation}
The set of all square-integrable sections together with the scalar product induced by (\ref{E:squareint}) is called the \emph{direct integral} of $(H_x)_{x\in X}$ and denoted by $\int^\oplus_X H_xm(dx)$.
\begin{remark}
If $\left\lbrace \eta_i:i\in\mathbb{N}\right\rbrace$ is a measurable field of orthonormal bases according to Remark \ref{R:ortho} and $\omega\in H= \int^\oplus_X H_xm(dx)$, then the sections $\omega_n$, given by $$\omega_n(x)=\sum_{i=0}^n (\omega(x),\eta_i(x))_{H_x}\eta_i(x)$$ approximate $\omega$ in $H$. A proof is given in \cite[Proposition II.1.6]{Dix}.
\end{remark} 

Given $a\otimes b\in\mathcal{A}\otimes \mathcal{B}_b$ with corresponding classes $[a\otimes b]_x\in\mathcal{H}_x$, the symbol $[a\otimes b]$ denotes the measurable section $x\mapsto [a\otimes b]_x$. Similarly for more general measurable sections $\omega$.

\begin{theorem}\label{T:coincide}\mbox{}
The Hilbert spaces $\mathcal{H}$ and $\int_X^\oplus \mathcal{H}_x m(dx)$ are isometrically isomorphic. In particular, for all $\omega, \eta\in\mathcal{H}$,
\[\left\langle\omega, \eta\right\rangle_{\mathcal{H}}=\int_X^\oplus \left\langle\omega,\eta\right\rangle_{\mathcal{H}_x}m(dx). \]
\end{theorem}
Consequently also $\mathcal{H}$ and $\int_X^\oplus \mathcal{B}_x m(dx)$ are isometrically isomorphic. In particular, up to an isomorphism, the definition of $1$-forms in \cite[Chapter 3 b) and Appendix D]{Eb99} arises as a special case of that in \cite{CS03,CS09}. 

\begin{proof}
For simple tensors $a\otimes b\in\mathcal{A}\otimes \mathcal{B}_b(X)$ set $\chi(a\otimes b):=[a\otimes b]$ and extend linearly to a mapping $\chi:\mathcal{A}\otimes \mathcal{B}_b(X)\to \int_X^\otimes (\omega,\eta)_{\mathcal{H}_x}m(dx)$. Since 
\[
\int_X\left\|[a\otimes b]_x\right\|_{\mathcal{H}_x}m(dx)=\int_Xb(x)^2 \left\|[a\otimes \mathbf{1}]_x\right\|_{\mathcal{H}_x}m(dx)=\int_Xb(x)^2\Gamma_x(a)m(dx)=\left\|a\otimes b\right\|_\mathcal{H}^2,\]
By denseness $\chi$ extends to an isometry from $\mathcal{H}$ into $\int_X^\otimes \mathcal{H}_x m(dx)$. To conclude surjectivity we make use of a totality argument from \cite[Theorem 7.3.11]{Eb99}. Suppose $\omega\in \int_X^\oplus \mathcal{H}_x m(dx)$ is such that 
\[0=\left\langle\omega,[a\otimes b]\right\rangle_{\mathcal{H}}=\int_X b(x)\left\langle\omega(x),[a\otimes \mathbf{1}]_x\right\rangle_{\mathcal{H}_x}m(dx)\]
for all $a\otimes b\in\mathcal{A}\otimes \mathcal{B}_b(X)$. Then in particular $\left\langle\omega(x),[a\otimes \mathbf{1}]_x\right\rangle_{\mathcal{H}_x}=0$ for all $a\in\mathcal{A}_0$ for $m$-a.e. $x$. But finite linear combinations $\sum_i\lambda_i [a_i\otimes \mathbf{1}]_x$ with functions $a_i\in\mathcal{A}_0$ and rational coefficients $\lambda_i$ are dense in the Hilbert space $\mathcal{H}_x$, therefore $\omega(x)=0$ for $m$-a.e. $x$ and consequently $\omega=0$ in $\int_X^\oplus \mathcal{H}_x m(dx)$. This implies that  the closure of the range $Im\:\chi$ of $\chi$ must be the entire direct integral.
\end{proof}

Let us agree upon the notation 
\begin{equation}\label{E:nicemapping}
\Gamma_{\mathcal{H},x}(\omega,\eta):=\left\langle\omega,\eta\right\rangle_{\mathcal{H}_x}\ \ \text{ for all $\omega, \eta\in\mathcal{H}$ and $x\in X$.} 
\end{equation}
Analogs of Lemma \ref{L:nonnegmeasure} and Corollary \ref{C:measure} now read as follows.

\begin{corollary}\mbox{}
\begin{enumerate}
\item[(i)] The measure $\Gamma_{\mathcal{H}}(\omega,\eta)$ from Lemma \ref{L:nonnegmeasure} is absolutely continuous with respect to $m$, and $\Gamma_{\mathcal{H},\cdot}(\omega,\eta)$ is a version of the Radon-Nikodym density $\frac{d\Gamma_{\mathcal{H}}(\omega,\eta)}{dm}$.
\item[(ii)] Definition (\ref{E:nicemapping}) provides a well defined and uniquely determined bilinear mapping $\Gamma_{\mathcal{H}}:\mathcal{H}\times\mathcal{H}\to L_1(X,m)$ such that for any $\omega\in\mathcal{H}$, $\Gamma_\mathcal{H,\cdot}(\omega)\geq 0$ $m$-a.e. 
\end{enumerate}
\end{corollary}

\begin{proof} (i) is obvious and (ii) is a simple consequence of Lemmas \ref{L:essentially} and \ref{L:equal}.
\end{proof}

\begin{corollary}\mbox{}
\begin{enumerate}
\item[(i)] If $\omega\in \mathcal{H}$ is such that $\left\|\omega\right\|_{\mathcal{H}}=0$,
then $\Gamma_{\mathcal{H},\cdot}(\omega)=0$ in $L_1(X,m)$. 
\item[(ii)] For $\omega, \eta\in\mathcal{H}$ with disjoint supports we have $\Gamma_{\mathcal{H}}(\omega,\eta)=0$ in $L_1(X,m)$. 
\end{enumerate}
\end{corollary}

As in \cite{CS03,CS09} a \emph{differential} $\partial:\mathcal{C}\to\mathcal{H}$ is defined by 
\[\partial(a)=a\otimes \mathbf{1} \ \ ,\ \ a\in\mathcal{C}.\]
The following properties are simple consequences of (\ref{E:scalarprodH}) and (\ref{E:left}). 

\begin{corollary}\label{C:simplecons}\mbox{}
\begin{enumerate}
\item[(i)] The operator $\partial$ is a derivation, i.e. it is linear and
\[\partial(fg)=(\partial f)g+f\partial g\ \ , \ \ f,g\in\mathcal{C}.\]
\item[(ii)] The operator $\partial$ is bounded, more precisely,
\[\mathcal{E}(f)\leq \left\|\partial f\right\|_{\mathcal{H}}^2\leq 2\mathcal{E}(f)\ \ , \ \ f\in\mathcal{C},\]
and if (\ref{E:desirable}) holds, we have $\left\|\partial f\right\|_{\mathcal{H}}^2=\mathcal{E}(f)$.
\end{enumerate}
\end{corollary}

On the other hand, Eberle \cite{Eb99} calls a linear map $d$ from an algebra $C$ into a direct integral $\int_X^\oplus H_xm(dx)$ of Hilbert spaces an \emph{$L_2$-differential} if 
\begin{enumerate}
\item[(i)] the span of $\left\lbrace fdg: f,g\in C\right\rbrace$ is dense in $\int_X^\oplus H_xm(dx)$ and
\item[(ii)] $\partial(fg)=fdg+gdf$, $f,g\in C$.
\end{enumerate}

Recall (\ref{E:EberleAnsatz}) and (\ref{E:EberleIso}). The following result is immediate.

\begin{corollary} The operator $\partial$ is an $L_2$-differential on $\mathcal{C}$. Given $f,g\in\mathcal{A}$, we have $[\partial f]_x=\Psi_x(d_xf)$ and
\[\left\langle \partial f,\partial g\right\rangle_{\mathcal{H}}=\int_X^ \oplus (d_xf,d_xg)_{\mathcal{B}_x}m(dx).\]
\end{corollary}

\begin{examples}\label{Ex:three}
In the situation of Examples \ref{Ex:one} (ii) and \ref{Ex:two} the space $\mathcal{H}$ is isometrically isomorphic to the Hilbert space $L_2(M,dvol, T^\ast M)$ of $L_2$-differential $1$-forms, and the restriction of $\partial$ to $C^\infty(M)$ coincides with the classical exterior derivative $d$. 
\end{examples}

\begin{remark}\mbox{}
\begin{enumerate}
\item[(i)] Similar assumptions as in \cite{Eb99} would allow to extend formula (\ref{E:EberleAnsatz}) to the entire algebra $\mathcal{C}$, such that each element $f\in\mathcal{C}$ can be assigned classes $d_xf\in\mathcal{B}_x$, $x\in X$. Then, if $df$ denotes the measurable vector field $x\mapsto d_xf$, $f\in\mathcal{C}$, the resulting mapping 
\[d:\mathcal{C}\to \int_X^\oplus \mathcal{B}_xm(dx)\]
defines an $L_2$-differential. In this case also (\ref{E:EberleIso}) extends to all of $\mathcal{C}$ and yields an isometry
$\Psi=\int_X^\oplus \Psi_xm(dx)$ taking  $\int_X\mathcal{B}_x m(dx)$ onto $\mathcal{H}$ such that $\partial=\Psi\circ d$. Note that this is closely related to the representation 
$$\mathcal{H}=L_2(X,m,(\mathcal{H}_x)_{x\in X})$$ 
discussed in detail in Sections \ref{S:gradient} and \ref{S:Sobo} below (see also Theorem~\ref{T:coincide}). 

\item[(ii)] If $(\mathcal{E},\mathcal{F})$ is strongly local, then we have $\Gamma_x(fg,h)=f(x)\Gamma_x(g,h)+g(x)\Gamma_x(f,h)$
for all $f,g,h\in\mathcal{C}$ by the Leibniz rule for energy measures \cite[Lemma 3.2.5]{FOT94}. This implies the \emph{localized Leibniz rule}
$d_x(fg)=f(x)d_xg+g(x)d_xf$. See for instance \cite[p. 151]{Eb99} or \cite[p. 112]{CS03}.

\item[(iii)] For the measurable field $(\mathcal{H}_x)_{x\in X}$ the function $x\mapsto d(x)=\dim \mathcal{H}_x$ from Remark \ref{R:ortho} coincides with the \emph{pointwise index} of $(\mathcal{E},\mathcal{F})$ as introduced by Hino in \cite{Hino10} (also related to the \emph{martingale dimension of fractals}, see \cite{Hino08}). There a detailed analysis of pointwise and global indices is provided and applied to first order derivatives of energy finite functions on a class of fractals.
\end{enumerate}
\end{remark}

\begin{remark}
The above construction has utilized the energy measures (\ref{E:energymeasure}) to generate a related algebraic structure. We would like to remind the reader of the well known fact that they also generate metric structures: Given a symmetric strongly local regular Dirichlet form $(\mathcal{E},\mathcal{F})$, consider
\begin{equation}\label{E:CC}
d(x,y):=\sup\left\lbrace f(x)-f(y): f\in\widetilde{\mathcal{C}}, \Gamma(f)\leq \mu\right\rbrace, \ \ \ x,y\in X,
\end{equation}
where $\widetilde{\mathcal{C}}$ is a core of $(\mathcal{E},\mathcal{F})$ and $\Gamma(f)\leq \mu$ stands for the requirement that $\Gamma(f)$ is absolutely continuous with respect to $\mu$ having density $\frac{\Gamma(f)}{d\mu}\leq 1$ $\mu$-a.e. Formula (\ref{E:CC}) provides a pseudo-metric $d$ on $X$, usually referred to as \emph{Carnot-Caratheodory distance} or \emph{intrinsic distance}. If $\widetilde{\mathcal{C}}$ separates the points of $X$, $d$ is a metric in the wide sense (i.e. satisfies the axioms of a metric but may attain the value $+\infty$). To our knowledge, (\ref{E:CC}) has first been considered in the context of Dirichlet forms in \cite{BM91,BM95,D89} and \cite{Sturm94,Sturm95}. Under the assumptions that $(X,d)$ is complete and the topology induced by $d$ on $X$ coincides with the original one, it had been shown in \cite{Sturm94} (together with \cite{Sturm95}) that $(X,d)$ is a geodesic space. In \cite{Stoll10} the completeness assumption had been dropped. Having in mind the constructions of the present paper, it would be interesting to know whether (or for which cores $\widetilde{\mathcal{C}}$) $(X,d)$ is a geodesic space without any further topological assumptions.
\end{remark}

\section{Vector fields, gradient and divergence}\label{S:gradient}

As a Hilbert space $\mathcal{H}$ is self-dual. We therefore regard $1$-forms also as \emph{vector fields}, exact $1$-forms $\partial f$ also \emph{gradients} and $\partial$ as the \emph{gradient operator}. As $\mathcal{C}$ is dense in $\mathcal{F}$ which in turn is dense in $L_2(X,\mu)$, $\partial$ may be viewed as densely defined unbounded operator 
\[\partial: L_2(X,\mu)\to\mathcal{H}\]
a priori equipped with the domain $dom\:\partial =\mathcal{C}$. As $(\mathcal{E},\mathcal{F})$ is a Dirichlet form, $\partial$ is closable by Corollary \ref{C:simplecons} and extends uniquely to a closed linear operator $\partial$ with domain $\mathcal{F}$.  

\begin{examples}
For the classical Dirichlet form on a smooth compact Riemannian manifold as in Examples \ref{Ex:one} (ii) and \ref{Ex:two} the operator $\partial$
then equals the closure in $L_2(M, dvol)$ of the exterior derivative $d:C^\infty(M)\to L_2(M,dvol, T^\ast M)$. 
\end{examples}

In the sequel we inquire about the adjoint $\partial^\ast$ of $\partial$. Let $\mathcal{C}^\ast$ denote the dual space of $\mathcal{C}$, normed by
\[\left\|w\right\|_{\mathcal{C}^\ast}=\sup\left\lbrace |w(f)|: f\in\mathcal{C}, \left\|f\right\|_{\mathcal{C}}\leq 1\right\rbrace \]
and automatically a Banach space. Given $f,g\in\mathcal{C}$, consider the mapping
\begin{equation}\label{E:divergence}
u\mapsto -\left\langle g\partial f, \partial u\right\rangle_\mathcal{H}=-\int_Xg\:d\Gamma(f,u)
\end{equation}
on $\mathcal{C}$. By Cauchy-Schwarz in $\mathcal{H}$ and Corollary \ref{C:simplecons} (ii) we have
\[|\left\langle g\partial f, \partial u\right\rangle_\mathcal{H}|\leq \sqrt{2}\left\|g\partial f\right\|_\mathcal{H} \mathcal{E}(u)^{1/2}\]
which says that (\ref{E:divergence}) defines an element $\partial^\ast(g\partial f)$ of $\mathcal{C}^\ast$ with norm bound
\[\left\|\partial^\ast(g\partial f)\right\|_{\mathcal{C}^\ast}\leq \sqrt{2}\left\|g\partial f\right\|_\mathcal{H}\ .\]
To 
\[\partial^\ast(g\partial f)=-\int_Xg\:d\Gamma(f,\cdot)\] 
we refer as the \emph{divergence of the vector field $g\partial f$}.
\begin{lemma}\label{L:div}
$\partial^\ast$ extends continuously to a bounded linear operator 
\[\partial^\ast:\mathcal{H}\to\mathcal{C}^\ast\]
with $\left\|\partial^\ast v\right\|_{\mathcal{C}^\ast}\leq \sqrt{2}\left\|v\right\|_{\mathcal{H}}$, $v\in\mathcal{H}$. Moreover,
\[\partial^\ast v(u)=-\left\langle v, \partial u\right\rangle_\mathcal{H}\]
for any $u\in\mathcal{C}$ and any $v\in\mathcal{H}$.
\end{lemma}
The operator $\partial^\ast$ will be called the \emph{divergence operator}. Note that this is a (global, non-local) definition in a \emph{distributional sense}.

\begin{proof}
For the application of the linear extension of $\partial^\ast$ to a finite linear combination $\sum_k g_k\partial f_k$ of simple vector fields we observe
\[|\partial^\ast(\sum_k g_k\partial f_k)(\varphi)|=|\sum_k \left\langle g_k\partial f_k,\partial \varphi\right\rangle_\mathcal{H}|\leq \sqrt{2}\left\|\sum_k g_k\partial f_k \right\|_\mathcal{H} \mathcal{E}(\varphi)^{1/2},\ \ \varphi \in \mathcal{C},\]
and since these elements form a dense subspace of $\mathcal{H}$, the lemma follows.
\end{proof}

In $X=\mathbb{R}^n$ we have the pointwise identity
\[div\:(g \: grad\:f)= g\Delta f +\nabla f \nabla g\]
for $f\in C^2(\mathbb{R}^n)$ and $g\in C^1(\mathbb{R}^n)$. Let $(L, dom\:L)$ denote the infinitesimal $L_2(X,\mu)$-generator of $(\mathcal{E},\mathcal{F})$. For $f\in dom\:L$ and $g,u\in \mathcal{C}$ we have 
\begin{equation}\label{E:goodcase}
(g L f)(u)=-\mathcal{E}(gu,f),
\end{equation}
and if $f\in\mathcal{C}$, we may use (\ref{E:goodcase}) as a definition of $gLf$: Since 
\[|(g L f)(u)|\leq \mathcal{E}(gu)^{1/2}\mathcal{E}(f)^{1/2}\leq \left\|u\right\|_{\mathcal{C}}\left\|g\right\|_{\mathcal{C}}\mathcal{E}(f)^ {1/2}\]
for any $u\in\mathcal{C}$ by Cauchy-Schwarz and (\ref{E:boundmult}), $g L f$ is a well defined member of $\mathcal{C}^\ast$. Similarly also the energy measure $\Gamma(f,g)$, seen as a linear functional 
\[\Gamma(f,g)(u):=\int_Xu\:d\Gamma(f,g)\]
on $\mathcal{C}$, is a member of $\mathcal{C}^\ast$, because $\left\|\Gamma(f)\right\|_{\mathcal{C}^\ast}\leq 2\mathcal{E}(f)$ and
a bound for $\Gamma(f,g)$ follows by polarization. 

\begin{lemma}\label{L:gpartf}
For any simple vector field $g\partial f$, $f,g\in\mathcal{C}$, we have
\begin{equation}\label{E:analog}
\partial^\ast (g\partial f)= g L f + \Gamma(f,g)\ ,
\end{equation}
seen as an equality in $\mathcal{C}^\ast$. If (\ref{E:desirable}) holds, we further have $L f=\partial^\ast \partial f$ for $f\in\mathcal{C}$.
\end{lemma}
\begin{proof}
This is now a simple consequence of the identity
\[-(g L f)(u)=\mathcal{E}(gu, f)=\int_Xgd\Gamma(u,f)+\int_Xu\:d\Gamma(f,g)\ ,\]
$u\in\mathcal{C}$, which itself may quickly be verified using (\ref{E:energymeasure}). The second statement follows because (\ref{E:desirable}) implies $\mathcal{E}(u,f)=-\left\langle u,\partial^\ast\partial f\right\rangle$.
\end{proof}

The preceding distributional definition can be complemented by a Hilbert space point of view. Generally the inclusions $\mathcal{C}\subset L_2(X,\mu)\subset \mathcal{C}^\ast$ are proper and seen as an operator
\[\partial^\ast: \mathcal{H}\to L_2(X,\mu),\]
the divergence $\partial^\ast$ is unbounded. As usual $v\in\mathcal{H}$ is said to be a member of $dom\:\partial^\ast$ if
there exists some (then automatically unique) $v^\ast\in L_2(X,\mu)$ such that $\left\langle u,  v^\ast\right\rangle_{L_2(X,\mu)}=-\left\langle \partial u, v\right\rangle_\mathcal{H}$ for all $u\in\mathcal{C}$.
In this case $\partial^\ast v:=v^\ast$ and 
\begin{equation}\label{E:ibp2}
\left\langle u, \partial^\ast v\right\rangle_{L_2(X,\mu)}=-\left\langle \partial u, v\right\rangle_\mathcal{H}\ , u\in \mathcal{C},
\end{equation}
i.e. $-\partial^\ast$ is the adjoint operator of $\partial$. It is immediate that $\left\lbrace \partial f: f\in dom\:L\right\rbrace\subset dom\:\partial^\ast$. As $-\partial^\ast$ is the adjoint of the densely defined and closable operator $\partial$ it is densely defined, see \cite{RS}. 

Probabilistic interpretations of $\partial$ and $\partial^\ast$ are discussed in Section \ref{S:stochcalc}.

\section{Applications to quasilinear PDE}\label{S:PDE1}

The discussed setup will now be used to solve PDE by fixed point and monotonicity arguments. We focus on equations involving terms $u\mapsto div\:a(grad\: u)$ and $u\mapsto b(\nabla u)$, where $a$ and $b$ are possibly nonlinear transformations. In our context these expressions rewrite $u\mapsto \partial^\ast(a(\partial u))$ and $u\mapsto b(\partial u)$, respectively.

Throughout this section we assume that $\mu$ is an admissible reference measure on $X$ and $(\mathcal{E},\mathcal{F})$ is a symmetric regular Dirichlet form on $L_2(X,\mu)$ satisfying (\ref{E:desirable}).

\emph{Quasilinear elliptic PDE in divergence form}. Consider the quasilinear PDE 
\begin{equation}\label{E:quasilinear}
\partial^\ast a(\partial u)=f.
\end{equation}
We study (\ref{E:quasilinear}) on the Hilbert space $L_2(X,\mu)$. The function $f$ is assumed to be an element of $L_2(X,\mu)$ and the gradient $\partial$ and divergence $\partial^\ast$ are interpreted as in Section \ref{S:gradient}. Let  
$Im\:\partial$ denote the image of $\mathcal{F}$ under $\partial$.

Assume that $a:\mathcal{H}\to\mathcal{H}$ satisfies the following monotonicity, growth and coercivity conditions: 
\begin{equation}\label{E:monotone}
\left\langle a(v)-a(w),v-w\right\rangle_\mathcal{H}\geq 0\ \ \text{for all $v,w\in Im\:\partial$},
\end{equation}
\begin{equation}\label{E:growth}
\left\|a(v)\right\|_{\mathcal{H}}\leq c_0(1+\left\|v\right\|_\mathcal{H})\ \ \ \text{for all $v\in Im\:\partial$}
\end{equation}
with some constant $c_0>0$,
\begin{equation}\label{E:coercive}
\left\langle a(v),v\right\rangle_\mathcal{H}\geq c_1\left\|v\right\|_\mathcal{H}^2-c_2 \ \ \text{for all $v\in Im\:\partial$}
\end{equation}
with constants $c_1>0$, $c_2\geq 0$. Finally, suppose the validity of a \emph{Poincar\'e inequality},
\begin{equation}\label{E:poincare}
\left\|f\right\|_{L_2(X,\mu)}^2\leq c_P\:\mathcal{E}(f)
\end{equation}
with some constant $c_P>0$ for all $f\in L_2(X,\mu)$ with $\int_Xfd\mu=0$. 
A function $u\in \mathcal{F}$ is called a \emph{weak solution to (\ref{E:quasilinear})} if 
\[\left\langle a(\partial u),\partial v\right\rangle_\mathcal{H}=-\left\langle f,v\right\rangle_{L_2(X,\mu)}\ \ \text{ for all $v\in \mathcal{F}$}.\]

The classical Brouwer-Minty monotonicity arguments based on Schauder's fixed point theorem, cf. \cite[Section 9.1]{Ev98}, now yield the following:
\begin{theorem}
Assume $a$ satisfies (\ref{E:monotone}), (\ref{E:growth}) and (\ref{E:coercive}) and suppose (\ref{E:poincare}) holds.
Then (\ref{E:quasilinear}) has a weak solution.  
Moreover, if $a$ is strictly monotone, i.e.
\begin{equation}\label{E:strictlymon}
\left\langle a(v)-a(w),v-w\right\rangle_\mathcal{H}\geq c_3\left\|v-w\right\|_\mathcal{H}^2 \ \ \text{ for all $v,w\in Im\:\partial$}
\end{equation}
with some constant $c_3>0$, then (\ref{E:quasilinear}) has a unique weak solution.
\end{theorem}
\begin{remark}
If $a$ is a decomposable (non-linear) operator, that is if $a=(a_x)_{x\in X}$ with $a_x:\mathcal{H}_x\to\mathcal{H}_x$, $x\in X$ and $m-\esssup_{x\in X}\left\|a_x\right\|_{\mathcal{H}_x\to\mathcal{H}_x}<\infty$, then to have (\ref{E:monotone}) it is sufficient to have
\[\left\langle a_x(v(x))-a_x(w(x))\right\rangle_{\mathcal{H}_x}\geq 0\]
for all $v,w\in Im\:\partial$ and $m$-a.e. $x\in X$. Similarly for conditions (\ref{E:growth}), (\ref{E:coercive}) and (\ref{E:strictlymon}).
\end{remark}

\emph{Quasilinear elliptic PDE in non-divergence form}. Consider the PDE
\begin{equation}\label{E:nondiv}
-Lu+ b(\partial u)+\varrho u=0,
\end{equation}
where $\varrho>0$ and $b$ is a generally non-linear function-valued mapping on $\mathcal{H}$. We assume that $b:\mathcal{H}\to L_2(X,\mu)$ is such that 
\begin{equation}\label{E:nondivgrowth}
\left\|b(v)\right\|_{L_2(X,\mu)}\leq c_4(1+\left\|v\right\|_\mathcal{H}),\ v\in Im\:\partial,
\end{equation}
with some $c_5>0$. A function $u\in \mathcal{F}$ is called a weak solution to (\ref{E:nondiv}) if 
\[\mathcal{E}(u,v)+\left\langle b(\partial u),\partial v\right\rangle_\mathcal{H}+\varrho\left\langle u,v\right\rangle_{L_2(X,\mu)}=0 \ \text{ for all $v\in \mathcal{F}$.}\] 
Along the lines of \cite[Section 9.2.2, Example 2]{Ev98}, we obtain the following.

\begin{theorem}
Assume that the embedding $\mathcal{F}\subset L_2(X,\mu)$ is compact and that (\ref{E:nondivgrowth}) holds. Then for any sufficiently large $\varrho>0$ there exists a weak solution to (\ref{E:nondiv}).
\end{theorem}
For convenience we briefly comment on the proof.
\begin{proof}
Given $u\in \mathcal{F}$, note that $-b(\partial u)\in L_2(X,\mu)$ and denote by $w$ the unique weak solution to the linear problem $-Lw+\varrho w=-b(\partial u)$, i.e. the unique function $w \in \mathcal{F}$ such that 
\begin{equation}\label{E:nondivweak}
\mathcal{E}(w,v)+\varrho\left\langle w,v\right\rangle_{L_2(X,\mu)}=-\left\langle b(\partial u),v\right\rangle_{L_2(X,\mu)}
\end{equation}
for all $v\in \mathcal{F}$. From (\ref{E:nondivgrowth}) we obtain $\left\|Lw\right\|_{L_2(X,\mu)}\leq c(1+\mathcal{E}_1(u)^{1/2})$. By the compact embedding, the mapping $u\mapsto \Phi(u):=w$ is easily seen to be continuous and compact from $\mathcal{F}$ into itself. See 
\cite[Section 9.2.2, Theorem 5]{Ev98}. The set 
\[\left\lbrace u\in \mathcal{F}: u=\lambda \Phi(u)\ \text{ for some $0<\lambda \leq 1$}\right\rbrace\] 
is bounded in $\mathcal{F}$: For a member of this set, (\ref{E:nondivweak}) implies
\begin{align}
\mathcal{E}(u)+\varrho\left\|u\right\|_{L_2(X,\mu)}^2&=-\lambda\left\langle b(\partial u),u\right\rangle_{L_2(X,\mu)}\notag\\
&\leq \left\| b(\partial u)\right\|_{L_2(X,\mu)}\left\|u\right\|_{L_2(X,\mu)}\notag\\
&\leq c_4\varepsilon(1+\left\|\partial u\right\|_\mathcal{H})\varepsilon^{-1}\left\|u\right\|_{L_2(X,\mu)}\notag\\
&\leq c_4(\varepsilon +\varepsilon\mathcal{E}(u)^{1/2}+\varepsilon^{-1}\left\|u\right\|_{L_2(X,\mu)})^2\notag\\
&\leq c(1+\varepsilon^ 2\mathcal{E}(u)+\varepsilon^{-2}\left\|u\right\|_{L_2(X,\mu)}^2)\notag
\end{align}
for any $\varepsilon>0$ and with a constant $c>0$ independent of $\varepsilon$, $\lambda$ and $u$. Now choose $\varepsilon>0$ sufficiently small and $\varrho>0$ sufficiently large to obtain $\mathcal{E}_1(u)^{1/2}\leq 2c$. Altogether this allows the application of Schaefer's fixed point theorem, cf. \cite[Section 9.2.2, Theorem 4]{Ev98}, to obtain the existence of a fixed point $u=\Phi(u)$ in $\mathcal{F}$.
\end{proof}

\section{Change of proper speed measure  and closability}\label{S:closability}

As before let $(\mathcal{E},\mathcal{F})$ be a symmetric regular Dirichlet form on $L_2(X,\mu)$, where $\mu$ is an admissible reference measure on $X$. Assume that $m$ is a measure satisfying Assumption \ref{A:energydominant}.
We will now address the closability of $(\mathcal{E},\mathcal{C})$ on $L_2(X,m)$, first in the case of $(\mathcal{E},\mathcal{F})$ irreducible or transient and then in the case that $(\mathcal{E},\mathcal{F})$ is induced by a regular resistance form.

A Dirichlet form $(\mathcal{E},\mathcal{F})$ is called \emph{transient} relative to $L_2(X,\mu)$ if there is a bounded $\mu$-integrable and $\mu$-a.e. positive function $\gamma$ on $X$  such that 
\[\int_X|u|\gamma d\mu\leq \mathcal{E}(u)^{1/2}\ \ \ \text{ for all $u\in\mathcal{F}$}.\]
Let $(T_t)_{t>0}$ denote the Markovian semigroup uniquely associated with $(\mathcal{E},\mathcal{F})$. The Dirichlet form $(\mathcal{E},\mathcal{F})$ is called \emph{irreducible} if any $(T_t)_{t>0}$-invariant set $A$ satisfies either $\mu(A)=0$ or $\mu(X\setminus A)=0$. 

\begin{theorem}\label{T:closabletransient}
Assume that $(\mathcal{E},\mathcal{F})$ is irreducible or transient. Then $(\mathcal{E},\mathcal{C})$ is closable on $L_2(X,m)$, and its closure $(\mathcal{E},\mathcal{F}^{(m)})$ is a symmetric local regular Dirichlet form. 
\end{theorem}

The \emph{$1$-capacity} associated with $(\mathcal{E},\mathcal{F})$ is defined as
\[\cpct(A):=\inf\left\lbrace \mathcal{E}_1(u): \text{$u\in\mathcal{F}$ and $u\geq 1$ $\mu$-a.e. on $A$}\right\rbrace \]
for open sets $A\subset X$. If the infimum is taken over the empty set, $\cpct(A)$ is set to be infinity. The $1$-capacity of an arbitrary subset $A\subset X$ is defined to be 
\[\cpct(A):=\inf\left\lbrace \cpct(B): \text{$B\supset A$, $B$ open}\right\rbrace.\]
If $(\mathcal{E},\mathcal{F})$ is transient we can also define the associated \emph{$0$-capacity} using $\mathcal{E}$ in place of $\mathcal{E}_1$, it will be denoted by $\cpct_0$. In this case a set has zero $0$-capacity if and only if it has zero $1$-capacity.

If $m$ does not charge sets of zero capacity then Theorem \ref{T:closabletransient} follows from Proposition \ref{P:fullqs} below by the arguments of \cite[Section 5, in particular Theorem 5.3]{FLJ}. See also \cite[Section 6.2, p.275]{FOT94} and \cite[Corollary 5.2.10]{ChFu12}. The transient case had already been established in \cite{RW}. If $m$ charges sets of zero capacity then it uniquely decomposes $m=m_0+m_1$ into an admissible reference measure $m_0$ that is absolutely continuous with respect to $\cpct$ and a nonnegative measure $m_1$ that is singular, see \cite{FST}. It is easy to see that as $m$ satisfies Assumption \ref{A:energydominant} also $m_0$ does. Closability with respect to $m_0$ implies closability with respect to $m$, hence Theorem \ref{T:closabletransient} holds also in this case.

A set $E\subset X$ is \emph{quasi-open} if for any $\varepsilon >0$ there exists an open set $G$ containing $E$ such that $\cpct(G\setminus E)=0$. A set is said to be \emph{quasi-closed} if it is the complement of a quasi-open set. A function on $X$ is called \emph{quasi-continuous} if for any $\varepsilon>0$ there exists an open set $G\subset X$ with $\cpct(G)<\varepsilon$ and the function is continuous on $X\setminus G$. Any element $u\in\mathcal{F}$ has an $m$-version that is quasi-continuous. See \cite[Theorem 2.1.3]{FOT94}. We will denote this version by $\widetilde{u}$. If a property holds on $X\setminus N$, where $N\subset X$ is a set of zero capacity, $\cpct(N)=0$, then we say this property holds \emph{quasi-everywhere}, abbreviated \emph{q.e.} For $A,B\subset X$ we write $A\subset B$ q.e. if $\cpct(A\setminus B)=0$. Given a nonnegative Radon measure $\nu$ on $X$ that charges no set of zero capacity, a set $\widetilde{F}\subset X$ is called a \emph{quasi-support} for $\nu$ if $\widetilde{F}$ is quasi-closed, $\nu(X\setminus \widetilde{F})=0$ and for any other set $\check{F}\subset X$ with these properties we have $\widetilde{F}\subset \check{F}$ q.e. The measure $\nu$ is said to have \emph{full quasi-support} if $X$ itself is a quasi-support for $\nu$. The following condition is necessary and sufficient for $\nu$ to have full quasi-support:
\begin{equation}\label{E:condqs}
\text{$\widetilde{u}=0$ $\nu$-a.e. if and only if $\widetilde{u}=0$ for any $u\in\mathcal{F}$.}
\end{equation}
A proof of this equivalence is given in \cite[Theorem 3.3]{FLJ}. Here we are interested in the quasi-supports of energy dominant measures.

\begin{proposition}\label{P:fullqs}
Assume $(\mathcal{E},\mathcal{F})$ is irreducible or transient and that $m$ does not charge sets of zero capacity. Then $m$ has full quasi-support.
\end{proposition}

To prove Proposition \ref{P:fullqs} we first establish a lemma. 
\begin{lemma}\label{L:zerolevel}
Let $u\in\mathcal{F}\cap L_\infty(X,\mu)$ be such that $\widetilde{u}=0$ $\Gamma(u)$-a.e. Then $\mathcal{E}(u)=0$.
\end{lemma}

The following short and elegant proof of this lemma was kindly suggested to us by one of the referees of this paper. 
\begin{proof}
For $\varepsilon>0$ define a function $h_\varepsilon:\mathbb{R}\to [0,1]$ by
\[h_\varepsilon(t):=\min\left\lbrace \frac{|t|}{\varepsilon}, 1\right\rbrace.\]
It is not difficult to see that for any $\varepsilon>0$ the function $t\mapsto \frac12 h_\varepsilon(t)t$ is a normal contraction, cf. \cite[p. 5]{FOT94}. Therefore $h_\varepsilon(u)u\in\mathcal{F}$ and $\sup_{\varepsilon}\mathcal{E}_1(h_\varepsilon(u)u)<+\infty$. Consequently there exists a sequence $(\varepsilon_k)_k$ converging to zero such that $(h_{\varepsilon_k}(u)u)_k$ converges $\mathcal{E}_1$-weakly to some $g\in\mathcal{F}$. By dominated convergence $(h_\varepsilon(u)u)_\varepsilon$ is seen to have the $L_2(X,\mu)$-limit $u$, hence $g=u$. On the other hand the defining identity (\ref{E:energymeasure}) for energy measures is valid also for functions from $\mathcal{F}\cap L_\infty(X,\mu)$, provided we take a quasi-continuous version of the integrand (see for instance \cite[Lemma 4.5.4]{FOT94}). This shows
\[2\mathcal{E}(h_\varepsilon(u)u,u)=\int_X h_\varepsilon(\widetilde{u})\:d\Gamma(u)+\int_X \widetilde{u}\:d\Gamma(h_\varepsilon(u),u).\]
By hypothesis the first integral on the right hand side vanishes, and using Cauchy-Schwarz also the second is seen to be zero. Therefore
\[\mathcal{E}(u)=\lim_k\mathcal{E}(h_{\varepsilon_k}(u)u,u)=0.\]
\end{proof}

\begin{remark} If $(\mathcal{E},\mathcal{F})$ is local, the condition $\widetilde{u}=0$ $\Gamma(u)$-a.e. is not needed. 
\end{remark}

We prove Proposition \ref{P:fullqs}.
 
\begin{proof} It suffices to check condition (\ref{E:condqs}). If $u\in\mathcal{F}$ is such that $\widetilde{u}=0$ q.e. then also $\widetilde{u}=0$ $m$-a.e. because $m$ does not charge sets of zero capacity. To verify the converse, let $u\in\mathcal{F}$ be such that $\widetilde{u}=0$ $m$-a.e. Then also $u_N:=\max\left\lbrace \min\left\lbrace \widetilde{u}, N\right\rbrace,-N\right\rbrace=0$ for any $N\in\mathbb{N}$ $m$-a.e. and therefore
\[\mathcal{E}(u)=\lim_N \mathcal{E}(u_N)=0\]
by the preceding lemma together with \cite[Theorem 1.4.2 (iii)]{FOT94}. Following \cite{FST} and \cite{KN91} set
\[\mathcal{E}^{m}(f,g):=\mathcal{E}(f,g)+\left\langle f,g\right\rangle_{L_2(X,m)}\]
for $f,g\in \widetilde{\mathcal{F}}\cap L_2(X,m)$, where $\widetilde{\mathcal{F}}$ denotes the collection of all $\mathcal{E}$-quasi-continuous versions of elements of $\mathcal{F}$. Then $(\mathcal{E}^{m},\widetilde{\mathcal{F}}\cap L_2(X,m))$ is a Dirichlet form on $L_2(X,\mu)$, see \cite[Lemma 6.1.1]{FOT94} and obviously $\mathcal{C}\subset \widetilde{\mathcal{F}}\cap L_2(X,m)$. Moreover, by \cite[Theorem 2.1 and Proposition 2.2]{KN91} the Dirichlet form $(\mathcal{E}^{m},\widetilde{\mathcal{F}}\cap L_2(X,m))$ is regular and transient. In particular, $f\mapsto \mathcal{E}^m(f)^{1/2}$ is a norm in $\widetilde{\mathcal{F}}\cap L_2(X,m)$, and since $\mathcal{E}^m(\widetilde{u})=0$ we obtain $\widetilde{u}=0$. If $(\mathcal{E},\mathcal{F})$ is transient the proof simplifies.
\end{proof}

Another case we are interested in arises if the regular Dirichlet form $(\mathcal{E},\mathcal{F})$ on $X$ is induced by a regular resistance form $(\overline{\mathcal{E}},\overline{\mathcal{F}})$ on the set $X$, which is then equipped with the topology determined by the associated resistance metric, see \cite{Ki01,Ki03} and in particular \cite[Definitions 3.1 and 9.5]{Ki12}. Regular resistance forms may for instance be obtained from regular harmonic structures on p.c.f. self-similar sets, \cite{Ki01}, on finitely ramified fractals (not necessarily self-similar) \cite{T08} and on some infinitely ramified sets such as Sierpinski carpets \cite{BB99}. A resistance form itself does not require the specification of a measure, and the conditions a measure must satisfy in order to obtain an induced Dirichlet form are rather weak. We quote the following result, which basically is a reformulation of \cite[Lemma 9.2 and Theorem 9.4]{Ki12}. 

\begin{theorem}\label{T:closableresistance}
Assume that $(\overline{\mathcal{E}},\overline{\mathcal{F}})$ is a regular resistance form on $X$ and that $X$, equipped with the associated resistance metric, is  locally compact, separable and complete. Assume further that $(\mathcal{E},\mathcal{F})$ is induced by $(\overline{\mathcal{E}},\overline{\mathcal{F}})$. Then for any admissible  $\nu\in\mathcal{M}_+(X)$ we have $\mathcal{C}=\overline{\mathcal{F}}\cap C_0(X)$, the form $(\overline{\mathcal{E}},\mathcal{C})$ is closable on $L_2(X,\nu)$, and its closure $(\mathcal{E},\mathcal{F}^{(\nu)})$ is a symmetric regular Dirichlet form.
\end{theorem}

\section{Sobolev spaces of functions and vector fields}\label{S:Sobo}

We will now introduce $L_p$-spaces of vector fields and related Sobolev spaces of functions. Throughout this section it is assumed that $(\mathcal{E},\mathcal{F})$ is a symmetric regular Dirichlet form, $m$ is a measure satisfying Assumption \ref{A:energydominant}, and $(\mathcal{E},\mathcal{C})$ is closable on $L_2(X,m)$.\\

For a measurable section $v=(v(x))_{x\in X}$ let 
\[\left\|v\right\|_{L_p(X,m,(\mathcal{H}_x)_{x\in X})}:=\left(\int_X\left\|v_x\right\|^p_{\mathcal{H}_x}m(dx)\right)^{1/p}\]
for $1\leq p<\infty$ and 
\[\left\|v\right\|_{L_\infty(X,m,(\mathcal{H}_x)_{x\in X})}:=\esssup_{x\in X}\left\|v_x\right\|_{\mathcal{H}_x}\] 
and define the spaces $L_p(X,m,(\mathcal{H}_x)_{x\in X})$, $1\leq p\leq\infty$ as the collections of the respective equivalence classes of $m$-a.e. equal sections having finite norm. By a variant of the classical pointwise Riesz-Fischer argument they form Banach spaces, separable for $1\leq p<\infty$. Note that $\mathcal{H}=L_2(X,m,(\mathcal{H}_x)_{x\in X})$. For $1< p< \infty$ and $1/p+1/q=1$ the H\"older inequality 
\begin{equation}\label{E:Hoelder}
\left|\int_X\left\langle v_x, w_x\right\rangle_{\mathcal{H}_x} m(dx)\right| 
\leqslant \left(\int_X\left\|v_x\right\|_{\mathcal{H}_x}^pm(dx)\right)^{1/p}\left(\int_X\left\|w_x\right\|_{\mathcal{H}_x}^qm(dx)\right)^{1/q}
\end{equation}
for $v\in L_p(X,m,(\mathcal{H}_x))$, $w\in L_q(X,m,(\mathcal{H}_x))$ follows from Cauchy-Schwarz in $\mathcal{H}$. We will write 
$\left\langle w,v\right\rangle$ for the the integral on the left hand side. 

If $f\in \mathcal{B}_b(X)$ and $v=(v(x))_{x\in X}\in L_p(X,m,(\mathcal{H}_x))$ then the product $fv$ is defined as the measurable section $x\mapsto f(x)v_x$, i.e. pointwise. Since
\[\left\|fv\right\|_{L_p(X,m,(\mathcal{H}_x)_x)}=\left(\int_X\left\|f(x)v_x\right\|_{\mathcal{H}_x}^pm(dx)\right)^{1/p}\leq \left\|f\right\|_{L_\infty(X,m)}\left\|v\right\|_{L_p(X,m,(\mathcal{H}_x)_x)}\]
the operation $v\mapsto fv$ is linear and bounded in $L_p(X,m,(\mathcal{H}_x))$ and is continuous with respect to the pointwise convergence of uniformly bounded sequences, i.e. if $\sup_n\left\|f_n\right\|_{L_\infty(X,m)}<\infty$ and $\lim_n f_n=f$ pointwise $m$-a.e. on $X$, then $\lim_n f_nv=fv$ in $L_p(X,m,(\mathcal{H}_x))$ for all $v\in L_p(X,m,(\mathcal{H}_x))$. For $p=2$ this multiplication coincides with (\ref{E:right}). 
 
We will make the following additional assumption.  
\begin{assumption}\label{A:core}
For any $1<p<\infty$ there is a space $\mathcal{C}_p \subset \mathcal{C}\cap L_p(X,m)$ such that 
\begin{enumerate}
\item[(COREI)] $\mathcal{C}_p$ is dense in $L_p(X,m)$, 
\item[(COREII)] $\mathcal{C}_p\otimes\mathcal{C}_p$ is dense in $L_p(X,m,(\mathcal{H}_x)_x)$ and
\item[(COREIII)] for all $f\in\mathcal{C}_p$, the energy measure $\Gamma(f)$ is absolutely continuous with respect to $m$ with density
\[\Gamma(f)=\frac{d\Gamma(f)}{dm} \in L_{p/2}(X,m).\]
\end{enumerate}
\end{assumption}
Let $\partial_p$ denote the restriction of $\partial$ to $\mathcal{C}_p$. By (COREIII) the operator $\partial_p$ maps $\mathcal{C}_p$ into $L_p(X,m,(\mathcal{H}_x)_x)$. Recall that in this section we assume the closability of $(\mathcal{E},\mathcal{C})$ on $L_2(X,m)$. As an immediate consequence $(\partial_2,\mathcal{C}_2)$ is seen to be a closable operator from $L_2(X,m)$ to $L_2(X,m,(\mathcal{H}_x))$, because also $(\mathcal{E},\mathcal{C}_2)$ is closable on $L_2(X,m)$. For $2<p<\infty$ we have the following result.

\begin{theorem}\label{P:closable}
Let the conditions of Assumption \ref{A:core} be valid and assume $m(X)<\infty$. Then $(\partial_p, \mathcal{C}_p)$ is a closable operator from $L_p(X,m)$ to $L_p(X,m,(\mathcal{H}_x))$ for any $2<p<\infty$.
\end{theorem}

\begin{proof} 
Let $(u_n)\subset \mathcal{C}_p$ be a sequence of functions converging to zero in $L_p(X,m)$ and such that $(\partial_p u_n)$ is Cauchy in $L_p(X,m,(\mathcal{H}_x))$. As the latter space is complete, a unique limit $v:=\lim_n \partial_p u_n\in L_p(X,m,(\mathcal{H}_x))$ exists. The measure $m$ being finite, $(u_n)_n$ is seen to be $\mathcal{E}$-Cauchy and convergent to zero in $L_2(X,\mu)$ what implies that $\mathcal{E}(u_n)$ goes to zero. For an arbitrary member $f\otimes g$ of $\mathcal{C}_q\otimes \mathcal{C}_q$ with $1/p+1/q=1$ we have 
\[\left\langle f\otimes g,v\right\rangle=\lim_n \left\langle f\otimes g,\partial_p u_n\right\rangle=\lim_n \left\langle f\otimes g,\partial u_n\right\rangle_\mathcal{H}=-\lim_n \partial^\ast(g\partial f)(u_n)=0\]
because
\[|\partial^ \ast(g\partial f)(u_n)|\leq \sqrt{2} \sup_{x\in X}|g(x)|\mathcal{E}(u_n)^{1/2}\mathcal{E}(f)^{1/2}.\]
By (COREII) therefore $\lim_n \partial_p u_n=0$ in $L_p(X,m,(\mathcal{H}_x))$. 
\end{proof}

For the rest of this section we take Assumption \ref{A:core} for granted and suppose that $2\leq p<\infty$ and $(\partial_p, \mathcal{C}_p)$ is closable. Its smallest closed extension is denoted by $(\partial_p, dom\:\partial_p)$, which then is a densely defined closed linear operator from $L_p(X,m)$ into $L_p(X,m,(\mathcal{H}_x))$. Note that for any simple vector field $g\partial f$ with $f,g\in\mathcal{C}_p$
we then have
\begin{equation}\label{E:simplep}
\left\|g\partial f\right\|_{L_p(X,m,(\mathcal{H}_x)_x)}=\left(\int_X |g(x)|^p \Gamma_x(f)^{p/2}m(dx)\right)^{1/p}.
\end{equation}
We write $H_0^{1,p}(X,m)$ for $dom\:\partial_p$, equipped with the norm 
\[\left\| u\right\|_{1,p}:=\left(\int_X\left(|u(x)|^p+\left\|\partial_x u\right\|_{\mathcal{H}_x}^p\right)m(dx)\right)^{1/p}, \ \ u\in H_0^{1,p}(X,m).\]
As $\left\|\cdot\right\|_{1,p}$ is equivalent to the graph norm of $\partial_p$, $H_0^{1,p}(X,m)$ is a closed subspace of $L_p(X,m)$, clearly Banach, and continuously embedded in $L_p(X,m)$. For $p=2$ we observe $H^{1,2}_0(X,m)=\mathcal{F}^{(m)}$.

Now the divergence operator $\partial^\ast$ may be seen as an unbounded operator $$\partial^\ast_q:L_q(X,m,(\mathcal{H}_x)) \to L_q(X,m),$$ where $1/p+1/q=1$, and similarly as in (\ref{E:ibp2}) we obtain an integration by parts formula by saying that an element $v\in L_q(X,m,(\mathcal{H}_x)_{x\in X})$ is in $dom\:\partial^\ast_q$ if there is some $v^\ast\in L_q(X,m)$ such that $\left\langle u,v^\ast\right\rangle=-\left\langle \partial u, v\right\rangle$ for all $u\in \mathcal{C}_p$. We write $\partial_q^\ast v:=v^\ast$ and 
\[\left\langle u, \partial_q^\ast v\right\rangle=-\left\langle \partial u, v\right\rangle\ , u\in \mathcal{C}_p.\]
By duality $dom\:\partial_q^\ast$ is then weakly dense in $L_q(X,m)$, cf. \cite{RS}. 

\begin{remark}
We provide a brief remark about related $p$-energies for $2\leq p<\infty$. The mapping 
\[f\mapsto \mathcal{E}_p(f):=\int_X\Gamma(f)^{p/2}dm, \ f\in H_0^{1,p}(X,m),\] 
is usually referred to as the \emph{$p$-energy functional}. One may define a functional of two arguments by 
\[\mathcal{E}_p(f,g):=\int_X\Gamma(f)^{p/2-1}\Gamma(f,g)dm,\ \ f,g\in H_0^{1,p}(X,m).\]
Note that $\mathcal{E}_p(f,f)=\mathcal{E}_p(f)$ and that by H\"older's inequality, $|\mathcal{E}_p(f,g)|\leq\mathcal{E}_p(f)^{(p-1)/p}\mathcal{E}_p(g)^{1/p}$. For functions $\varphi,\psi\in dom\:L^{(m)}$ we observe $\mathcal{E}_p(\varphi,\psi):=\frac{1}{p}\frac{d}{dt}\mathcal{E}_p(\varphi+t\psi)|_{t=0}$.
A generalized $p$-Laplacian may be defined in the weak sense by associating to $f\in H_0^{1,p}(X,m)$ the element $\Delta_pf$ of the dual space $(H_0^{1,p}(X,m))^\ast$ given by $(\Delta_p f)(g):=-\mathcal{E}_p(f,g)=-\left\langle \left\|\partial_\cdot f\right\|_{\mathcal{H}_\cdot}^{p-2}\partial f,\partial g\right\rangle_\mathcal{H}$,
$g\in H_0^{1,p}(X,m)$. Integrating by parts we obtain $\Delta_p f=\partial_p^\ast\left(\left\|\partial_{\cdot}f\right\|_{\mathcal{H}_\cdot}^{p-2}\partial f\right)$. If $L=\Delta$ is the classical Laplacian on $\mathbb{R}^n$ and $m(dx)=dx$ the $n$-dimensional Lebesgue measure, then $\Delta_p$ is the usual $p$-Laplacian.

Another definition for a $p$-energy on Sierpinski gasket type fractals had been proposed in \cite{HPS04}. It had been constructed by solving an abstract renormalization problem whose solution allows to define a $p$-energy as the limit of an rescaled sequence of discrete $p$-energies on 
approximating graphs. A related $p$-Laplacian had been investigated in \cite{StrW04}. However, it is not difficult to see that the energy rescaling is different and therefore the domains of this $p$-energy and the one defined above will generally be disjoint.
\end{remark}

\section{Existence of continuous coordinates}\label{S:coords}

One possible way to verify Assumption \ref{A:core} in a non-classical contexts is to use abstract continuous coordinates. Let $(\mathcal{E},\mathcal{F})$ be a symmetric local regular Dirichlet form and $m$ is a measure according to Assumption \ref{A:energydominant}. For the measure $\widetilde{m}$, as constructed in Lemma \ref{L:lem}, we will actually prove the existence of coordinates. We will now work under the following additional assumption:

\begin{assumption}\label{A:goodmeasure} In addition to Assumption~\ref{A:energydominant}
we assume that the measure $m$ is finite and does not charge sets of zero capacity.
\end{assumption}

Let $\left\lbrace \varphi_i\right\rbrace_{i\in I}\subset \mathcal{C}$ be a set of functions indexed by some set $I\neq \emptyset$. We say that $\left\lbrace \varphi_i\right\rbrace_{i\in I}$ is a set of \emph{continuous coordinates} for $\mathcal{E}$ with respect to $m$ if the following conditions are satisfied: 
\begin{enumerate}
\item[(COI)] for all $i,j\in I$, $\Gamma(\varphi_i,\varphi_j)\in L_1(X,m)\cap L_\infty(X,m)$,
\item[(COII)]  The space $\mathcal{F}C_b^1(X,\left\lbrace \varphi_i\right\rbrace)$ of all cylinder functions of the form
\[f=F(\varphi_{i_1},...,\varphi_{i_m}), \ \ i_1,...,i_k\in I\]
with suitable $k\in\mathbb{N}$ and $F\in C_b^1(\mathbb{R}^k)$, $F(0)=0$, is dense in $\mathcal{C}$ with respect to the norm in $\mathcal{F}$, i.e $\mathcal{E}_1$-dense.
\end{enumerate}
 
For cylinder functions $f=F(\varphi_{i_1},...\varphi_{i_m})$ and $g=G(\varphi_{j_1},...,\varphi_{j_n})$ with $F\in C_b^1(\mathbb{R}^m)$ and $G\in C_b^1(\mathbb{R}^n)$ satisfying $F(0)=G(0)=0$, we then have
\[\Gamma(f,g)=\sum_{k=1}^m\sum_{l=1}^n\frac{\partial F}{\partial x_k}(\varphi_{i_1},...,\varphi_{i_m})\frac{\partial G}{\partial x_l}(\varphi_{j_1},...,\varphi_{j_n})\Gamma(\varphi_k,\varphi_l)\]
by the chain rule, cf. \cite[Theorem 3.2.2]{FOT94}. In particular, $\Gamma(f,g)$ is a member of $L_\infty(X,m)$ and has compact support.

From (COII) we can obtain further approximation and denseness results.

\begin{lemma}\label{L:boundedpointwise}
Suppose that $m$ satisfies Assumption \ref{A:goodmeasure} and that $\left\lbrace \varphi_i\right\rbrace_{i\in I}$ is a set of continuous coordinates for $\mathcal{E}$ with respect to $m$. Then any function $g\in \mathcal{C}$ can be approximated pointwise $m$-a.e. by a uniformly bounded sequence of functions from $\mathcal{F}C_b^1(X,\left\lbrace \varphi_i\right\rbrace)$.
\end{lemma}

\begin{proof}
Let $g\in \mathcal{C}$. By (COII) there is a sequence $(g_n)_n\subset \mathcal{F}C_b^1(X,\left\lbrace \varphi_i\right\rbrace)$ converging to $g$ in $\mathcal{F}$ with respect to the $\mathcal{E}_1$-norm. Switching to a subsequence if necessary we may assume $(g_n)_n$ also converges to $g$ q.e. by \cite[Theorem 2.1.4]{FOT94}. As $m$ does not charge sets of zero capacity, $g$ is also the $m$-a.e. pointwise limit of $(g_n)_n$. Now set \[s:=\sup_{x\in X}|g(x)|\] 
and let $\phi\in C_b^1(\mathbb{R})$ be a monotone function that satisfies 
\[\phi(y)=
\begin{cases}
-2s \ \text{ if $y<-2s$}\\
y \ \text{ if $-s\leq y\leq s$}\\
2s \ \text{ if $y>2s$}.
\end{cases}\]
Note that $\phi(g_n)\in \mathcal{F}C_b^1(X,\left\lbrace \varphi_i\right\rbrace)$ for any $n$ and $\sup_n\sup_{X}|\phi(g_n)|\leq 2s$. Also the functions $\phi(g_n)$ converge to $g$ $m$-a.e. pointwise.
\end{proof}

\begin{corollary} Suppose that $m$ satisfies Assumption \ref{A:goodmeasure} and that $\left\lbrace \varphi_i\right\rbrace_{i\in I}$ is a set of continuous coordinates for $\mathcal{E}$ with respect to $m$. Then for any $1< p<\infty$ conditions (COREI) and (COREIII) in Assumption \ref{A:core} are satisfied with $\mathcal{C}_p=\mathcal{F}C_b^1(X,\left\lbrace \varphi_i\right\rbrace)$.
\end{corollary}

\begin{proof}
(COREIII) follows directly from (COI). Hence it suffices to prove that $\mathcal{F}C_b^1(X,\left\lbrace \varphi_i\right\rbrace)$ is dense in $L_p(X,m)$, $1<p<\infty$. By the denseness of $C_0(X)$ in $L_p(X,m)$, the finiteness of $m$ and the regularity of $(\mathcal{E},\mathcal{F})$ it is enough to show any function $g\in \mathcal{C}$ can be approximated in $L_p(X,m)$-norm by a sequence of functions from $\mathcal{F}C_b^1(X,\left\lbrace \varphi_i\right\rbrace)$. As $m$ is finite, this follows from Lemma \ref{L:boundedpointwise}.
\end{proof}

Another consequence concerns the spaces $L_p(X,m,(\mathcal{H}_x)_x)$.
\begin{lemma}
Suppose that $m$ satisfies Assumption \ref{A:goodmeasure} and that $\left\lbrace \varphi_i\right\rbrace_{i\in I}$ is a set of continuous coordinates for $\mathcal{E}$ with respect to $m$. Then for any $1<p<\infty$, 
\[\mathcal{S}:=\lin\left\lbrace g\partial f: f,g\in\mathcal{F}C_b^1(X,\left\lbrace \varphi_i\right\rbrace)\right\rbrace\] is a dense subspace of $L_p(X,m,(\mathcal{H}_x)_x)$. In particular, condition (CORE II) in Assumption \ref{A:core} holds with $\mathcal{C}_p=\mathcal{F}C_b^1(X,\left\lbrace \varphi_i\right\rbrace)$.
\end{lemma}

\begin{proof}  
Consider the space 
\[\mathcal{S}_0:=\lin\left\lbrace g\partial f: f\in\mathcal{F}C_b^1(X,\left\lbrace \varphi_i\right\rbrace), g\in\mathcal{B}_b(X)\right\rbrace,\]
it obviously contains $\mathcal{S}$. By (\ref{E:simplep}) it is easily seen that for any $1<p<\infty$, $\mathcal{S}_0$ is a subspace of $L_p(X,m,(\mathcal{H}_x)_x)$. We will prove it is dense. 

The space $\mathcal{S}_0$ is dense in the Hilbert space $\mathcal{H}=L_2(X,m,(\mathcal{H}_x)_x)$: By the definition of $\mathcal{H}$, it suffices to approximate finite linear combinations $\sum_i a_i\otimes b_i\in\mathcal{C}\otimes \mathcal{B}_b(X)$. For fixed $i$, let $(a_i^{(m)})_m\subset\mathcal{F}C_b^1(X,\left\lbrace \varphi_i\right\rbrace)$ be a sequence approximating $a_i$ in $\mathcal{E}$. Then 
\[\left\|\sum_i a_i\otimes b_i -\sum_i a_i^{(m)}\otimes b_i\right\|_\mathcal{H}^2\leq c\sum_i\int_X b_i^2d\Gamma(a_i-a_i^{(m)})\]
with a constant $c>0$ that depends only on $\sum_i a_i\otimes b_i$. The right hand side is bounded by $2c(\max_i\sup_X b_i^2) \sum_i\mathcal{E}(a_i-a_i^{(m)})$ and therefore converges to zero as $m$ goes to infinity. 

We will now use a duality argument to prove that $\mathcal{S}_0$ is also dense in $L_p(X,m,(\mathcal{H}_x)_x)$: Assume it were not, then by Hahn-Banach we could find some $\eta\in L_q(X,m,(\mathcal{H}_x)_x)$, $1/p+1/q=1$, such that $\left\|\eta\right\|_{L_q(X,m,(\mathcal{H}_x)_x)}=1$ and 
\begin{equation}\label{E:kill}
\left\langle \omega,\eta\right\rangle=0 \ \ \text{ for all $\omega\in \mathcal{S}_0$.}
\end{equation}
Using $\eta$ we can construct an element $\omega$ of $\mathcal{S}_0$ for which (\ref{E:kill}) fails to hold. First of all, we deal with integrability issues by cutting off and approximating. For any $N\in\mathbb{N}$ let $K_N\subset X$ be compact such that $m(X\setminus K_N)<1/N$ and set
\[S_N:=\left\lbrace x\in X: \left\|\eta_x\right\|_{\mathcal{H}_x}<N\right\rbrace\cap K_N.\]
Then 
\[\lim_N\left\|\eta\mathbf{1}_{S_N}-\eta\right\|_{L_q(X,m, (\mathcal{H}_x)_x)}^q=\lim_N\int_X|\mathbf{1}_{S_N}(x)-\mathbf{1}(x)|^q\left\|\eta_x\right\|_{\mathcal{H}_x}^qm(dx)=0\]
by dominated convergence and accordingly for fixed $\varepsilon>0$ there exists some $N_\varepsilon>0$ such that for any $N\geq N_\varepsilon$, $\left\|\eta\mathbf{1}_{S_N}-\eta\right\|_{L_q(X,m,(\mathcal{H}_x)_x)}<\varepsilon$. Note also that
$\eta\mathbf{1}_{S_N}\in\mathcal{H}$ for all $N\in\mathbb{N}$ since
\[\int_X\left\|\eta_x\mathbf{1}_{S_N}(x)\right\|_{\mathcal{H}_x}^2m(dx)<N^2m(K_N)<\infty.\]
As $\left\|\eta\mathbf{1}_{S_N}\right\|_{L_q(X,m,(\mathcal{H}_x)_x)}>1-\varepsilon$, the function $x\mapsto \mathbf{1}_{S_N}(x)\left\|\eta_x\right\|_{\mathcal{H}_x}$ cannot be zero $m$-a.e. hence also
\[\delta_N:=\left\|\eta\mathbf{1}_{S_N}\right\|_\mathcal{H}>0.\]
Now let $(\omega_n)_n\subset \mathcal{S}_0$, $\omega_n=\sum_i a_i^{(n)}\otimes b_i^{(n)}$ be a sequence that approximates $\eta\mathbf{1}_{S_N}$ in $\mathcal{H}$. Let $0<\gamma<\delta_N$ and $n\in\mathbb{N}$ be so large that  
\[\left\|\eta\mathbf{1}_{S_N}-\omega_n\right\|_{\mathcal{H}}\leq \gamma.\]
Since $|\left\langle \mathbf{1}_{S_N}\eta,\mathbf{1}_{S_N}\eta-\omega_n\right\rangle|\leq \gamma\delta_N$ we obtain
\begin{equation}\label{E:contra}
|\left\langle \mathbf{1}_{S_N}\omega_n,\eta \right\rangle|=|\left\langle \omega,\mathbf{1}_{S_N}\eta \right\rangle|>\delta_N(\delta_N-\gamma)>0.
\end{equation}
On the other hand
\[\omega:=\mathbf{1}_{S_N}\omega_n=\sum_i \mathbf{1}_{S_N}(a_i^{(n)}\otimes b_i^{(n)})=\sum_i a_i^{(n)}\otimes (\mathbf{1}_{S_N}b_i^{(n)})\]
itself is an element of $\mathcal{S}_0$. Therefore (\ref{E:contra}) contradicts (\ref{E:kill}), and $\mathcal{S}_0$ must be dense in $L_p(X,m,(\mathcal{H}_x)_x)$.

Finally, note that the space
$\mathcal{S}$ is $L_p(X,m,(\mathcal{H}_x)_x)$-dense in $\mathcal{S}_0$: Let $\sum_i a_i\otimes b_i\in \mathcal{F}C_b^1(X,\left\lbrace \varphi_i\right\rbrace)\otimes \mathcal{B}_b(X)$. Any $b_i$ can be approximated uniformly by a sequence from $C_0(X)$, and by the regularity of $(\mathcal{E},\mathcal{F})$ together with Lemma \ref{L:boundedpointwise}, any $b_i$ can be approximated pointwise by a uniformly
bounded sequence $(b_i^{(m)})_m$ of functions from $\mathcal{F}C_b^1(X,\left\lbrace \varphi_i\right\rbrace)$. By (\ref{E:simplep}),
\[\left\|\sum_i a_i\otimes b_i-\sum_i a_i\otimes b_i^{(m)}\right\|_{L_p(X,m,(\mathcal{H}_x)_x)}\leq c\sum_i \left(\int_X|b_i(x)-b_i^{(m)}(x)|^p\Gamma_x(a_i)^{p/2}m(dx)\right)^{1/p}\]
(with $c>0$ depending on $\sum_i a_i\otimes b_i$), which converges to zero since $m$ is finite and $\Gamma_\cdot(a_i)\in L_\infty(X,m)$ for any $i$.
\end{proof}

In the following sense the existence of a countable set of continuous coordinates is always guaranteed. Recall that $\widetilde{m}$ denotes the measure (\ref{E:mtilde}) constructed in Lemma \ref{L:lem} as a sum of the energy measures of certain functions $f_n\in\mathcal{C}$, $n\in\mathbb{N}$, considered in (\ref{E:fn}).

\begin{theorem}\label{T:coordexist}
Let $(\mathcal{E},\mathcal{F})$ be a symmetric local regular Dirichlet form and $\widetilde{m}$ the measure given by (\ref{E:mtilde}). Then the set $\left\lbrace f_n\right\rbrace_{n\in\mathbb{N}}$ of functions $f_n$ according to (\ref{E:fn}) is a set of continuous coordinates for $\mathcal{E}$ with respect to $\widetilde{m}$.
\end{theorem}

\begin{proof}
Obviously
\[\frac{d\Gamma(f_n)}{d\widetilde{m}}\leq 2^n\ \ \text{$\widetilde{m}$-a.e}\]
for any $n$, and by Cauchy-Schwarz (COI) follows. By construction $\lin\:\left(\left\lbrace f_n\right\rbrace_n\right)$ is $\mathcal{E}_1$-dense
in $\mathcal{C}$, what implies (COII). Finally, recall that the finite measure $\widetilde{m}$ does not charge sets of zero capacity, \cite[Lemma 3.2.4]{FOT94}.
\end{proof}

We give an non-classical example.

\begin{examples} Let $(\overline{\mathcal{E}},\overline{\mathcal{F}})$ be a self-similar resistance form, \cite{Ki01,Ki03, Ki12}, on a self-similar finitely ramified fractal $X$, see \cite[Definitions 7.1 and 7.4]{T08}. We assume that $(\overline{\mathcal{E}},\overline{\mathcal{F}})$ is regular in the sense of \cite[Section 7]{T08}. (This notion of regularity is different from the one addressed in \cite{Ki12} and in Theorem \ref{T:closableresistance}). Let $\varphi_1,...,\varphi_k$ be a complete, up to constants, energy orthonormal set of harmonic functions and define a finite reference measure by
\begin{equation}\label{E:Kusuokameasure}
m:=\sum_{j=1}^k \Gamma(\varphi_j)
\end{equation}
(\emph{Kusuoka measure}), where $\Gamma(\varphi_j)$ are the energy measures of the functions $\varphi_j$. Consider the map $\psi:X\to\mathbb{R}^m$, $\psi(x)=(\varphi_1(x),...,\varphi_k(x))$, cf. \cite{T08}. We assume that $\psi: X\to \psi(X)$ is a homeomorphism.
This implies that all $\varphi_j$ are continuous on $X$, cf. \cite[Proposition 5.3]{T08}. By \cite[Theorems 3 and 10]{T08}  $(\overline{\mathcal{E}},\overline{\mathcal{F}})$ induces a local regular Dirichlet form $(\mathcal{E},\mathcal{F}^{(m)})$ on $L_2(X,m)$. Using Theorem \ref{T:coordexist} we may now conclude that $\left\lbrace \varphi_i\right\rbrace_{i}$ is a set of continuous coordinates for $\mathcal{E}$ with respect to $m$.
\end{examples}

\begin{remark}\label{R:example} An alternative argument to prove at least (COI) may be obtained directly using the associated generator.
Here we assume that $(\mathcal{E},\mathcal{C})$ is closable on $L_2(X,m)$. Let $L^{(m)}$ be the infinitesimal generator of the Dirichlet form $(\mathcal{E},\mathcal{F}^{(m)})$ in $L_2(X,m)$ and $dom\:L^{(m)}$ its domain. By $(L^{(m),1}, dom\:L^{(m),1})$  we denote the closure in $L_1(X,m)$ of $(L^{(m)}, dom\:L^{(m)})$, see \cite[Section I.2.4]{BH91} and recall that we have assumed $m$ to be finite.
As all energy measures $\Gamma(f)$, $f\in\mathcal{C}$, are absolutely continuous with respect to $m$, \cite[Theorems I.4.2.1 and I.4.2.2]{BH91} tell that $dom\:L^{(m),1}\cap L_\infty(X,m)$ is an algebra and 
\[\Gamma(\varphi,\psi)=L^{(m),1}(\varphi\psi)-\varphi L^{(m)}\psi-\psi L^{(m)}\varphi\] 
for all $\varphi,\psi\in dom\:L^{(m)}$. Therefore, if $\left\lbrace \varphi_i\right\rbrace_{i\in I}\subset dom\:L^{(m)}$ and for all $i,j\in I$, the functions $L^{(m)}\varphi_i$ and $L^{(m),1}(\varphi_i\varphi_j)$ are members of $L_\infty(X,m)$, condition (COI) is obviously satisfied.
\end{remark}

\section{Applications to SPDE}\label{S:PDE}

The results of Section \ref{S:Sobo} may for instance be used to study deterministic or stochastic evolution equations in the variational framework. To discuss a class of examples we assume throughout the entire section that $2\leq p<\infty$, $m$ is finite, (\ref{E:desirable}) holds and $(\partial_p,\mathcal{C}_p)$ is closable on $L_p(X,m)$.\\

\emph{SPDE in variational form}. SPDE in variational form have been studied first in \cite{KR} and \cite{Par}, a brief exposition is given in \cite{PrevRoeck}. For simplicity consider It\^o SPDE with additive Brownian noise of type
\begin{equation}\label{E:SPDE}
du(t)=\partial^\ast a(\partial u(t)) dt +\sqrt{Q} dW(t)
\end{equation}
on $(0,T)\times X$ with some initial condition $u(0)=u_0$. Here $(W(t))_{t\geq 0}$ is a cylindrical Wiener process on $L_2(X,m)$ of the form
\[W(t)=\sum_{k=1}^\infty \beta_k(t)e_k\]
where $(\beta_k)_k$ is a sequence of mutually independent one-dimensional standard Brownian motions on a filtered complete probability space $(\Omega,\mathcal{F},(\mathcal{F}_t)_{t\geq 0},\mathbb{P})$, and $\left\lbrace e_k\right\rbrace_k$ is an orthonormal basis in $L_2(X,m)$. Then $L_p(X,m)\subset L_2(X,m)$, and we have 
\[H^{1,p}_0(X,m)\subset L_2(X,m)\subset (H_0^{1,p}(X,m))^\ast\ ,\]
where as before $(H_0^{1,p}(X,m))^\ast$ denotes the dual space of $H^{1,p}_0(X,m)$, and the embeddings are continuous. We write \[\left\langle u,v\right\rangle:=v(u)\ \ \ \text{ for $u\in (H_0^{1,p}(X,m))^\ast$ and $v\in H^{1,p}_0(X,m)$}.  \]
Generalizing the growth condition (\ref{E:growth}) we may require $a$ to be a bounded operator 
\[a:L_p(X,m,(\mathcal{H}_x)_x)\to L_q(X,m,(\mathcal{H}_x)_x)\]
with $1/p+1/q=1$ and such that 
\begin{equation}\label{E:growthp}
\left\|a(v)\right\|_{L_q(X,m,(\mathcal{H}_x)_x)}\leq c_0(1+\left\|v\right\|_{L_p(X,m,(\mathcal{H}_x)_x)}^{p-1})
\end{equation}
for all $v\in L_p(X,m,(\mathcal{H}_x)_x)$.
\begin{remark}
(\ref{E:growthp}) is obviously valid with $p=2$ if $a=(a_x)_x$ with bounded operators $a_x:\mathcal{H}_x\to \mathcal{H}_x$ such that m-$\esssup_{x\in X}\left\|a_x\right\|_{\mathcal{H}_x\to\mathcal{H}_x}<\infty$.
\end{remark}
The following is a simple consequence of the H\"older inequality (\ref{E:Hoelder}):

\begin{lemma}
If $a$ satisfies (\ref{E:growthp}) then $\partial^\ast a(\partial\cdot)$ defines a bounded operator from
$H^{1,p}_0(X,m)$ into $(H^{1,p}_0(X,m))^\ast$ with 
\[\left\|\partial^\ast a(\partial u)\right\|_{(H^{1,p}_0(X,m))^\ast}\leq c_6(1+\left\|u\right\|_{H^{1,p}_0(X,m)}^{p-1})\ ,\]
$u\in H^{1,p}_0(X,m)$, with a constant $c_6>0$.
\end{lemma}

Similarly as in the case of (\ref{E:quasilinear}) we may invoke a general solution theory \cite{KR,Par}, provided some regularity conditions are satisfied. In addition to (\ref{E:growthp}) we will require the versions
\begin{equation}\label{E:coercivep}
\left\langle a(\partial f),\partial f\right\rangle\geq c_1\left\|f\right\|_{1,p}^p-c_2 \left\|f\right\|_{L_2(X,m)}^2\ \ \text{for all $f\in H^{1,p}_0(X,m)$}
\end{equation}
with constants $c_1>0$, $c_2\geq 0$ and
\begin{equation}\label{E:monotonep}
\left\langle a(\partial f)-a(\partial g),\partial f-\partial g\right\rangle\geq c_3\left\|f-g\right\|_{L_2(X,m)}^2 \ \ \text{for all $f,g\in H^{1,p}_0(X,m)$},
\end{equation}
with $c_3>0$ of (\ref{E:coercive}) and (\ref{E:strictlymon}) with the left hand sides interpreted in the sense of duality. Finally, assume that for all $u,v,w$ from the image $Im\:\partial_p$ of $H_0^{1,p}(X,m)$ under $\partial_p$,
\begin{equation}\label{E:hemicont}
\text{ the function }\ \ \lambda\mapsto \left\langle a(u+\lambda v),w \right\rangle \ \ \text{ is continuous at zero.}
\end{equation}
\begin{remark}\mbox{}
Note that if (\ref{E:growthp}) is valid and $a=(a_x)_x$ is decomposable as before, the relation 
\[\lim_{\lambda\to 0} \left\langle a_x(u(x)+\lambda v(x)),z(x)\right\rangle_{\mathcal{H}_x}=\left\langle a_x(v(x)),z(x)\right\rangle_{\mathcal{H}_x}\]
for $m$-a.e. $x\in X$ implies (\ref{E:hemicont}), because
\[|\left\langle a(\partial f+\lambda \partial g), \partial h\right\rangle|\leq c(1+\left\|f\right\|_{1,p}+\left\|g\right\|_{1,p})\left\|h\right\|_{1,p}, \ \ f,g,h\in H^{1,p}_0(X,m),\]
as one can easily verify.
\end{remark}

A continuous $(\mathcal{F}_t)$-adapted process $u=(u(t))_{t\in [0,T]}$ is called \emph{a solution to (\ref{E:SPDE}) with initial condition $u_0$} if 
\[\mathbb{E}\int_0^T(\left\|u(t)\right\|_{1,p}^p+\left\|u(t)\right\|_{L_2(X,m)}^2)dt<\infty\]
and 
\[u(t)=u_0+\int_0^t \partial^\ast a(\partial \widetilde{u}(s))ds+\int_0^t\sqrt{Q}dW(s)\ ,\ \ t\in [0,T],\]
seen as an identity of $(H^{1,p}_0(X,m))^\ast$-valued functions, where $\widetilde{u}$ is any $H^{1,p}_0(X,m)$-valued progressively measurable $dt\otimes d\mathbb{P}$-version of $u$. 

The following is a special case of the classical results in \cite{KR,PrevRoeck}.
\begin{theorem}
Let $2\leq p<\infty$, $m(X)<\infty$ and assume that $a$ satisfies (\ref{E:growthp}), (\ref{E:coercivep}), (\ref{E:monotonep}) and (\ref{E:hemicont}). Let 
\[\mathbb{E}\int_X u_0^2(x) m(dx)<\infty.\]
Then (\ref{E:SPDE}) has a unique solution $u$ with initial condition $u_0$.
\end{theorem} 

\begin{examples} 
A specific example is given by the following \emph{stochastic $p$-Laplace equation}: Let $a=(a_x)_{x\in X}$ with
\[a_x(v(x)):=\left\|v(x)\right\|_{\mathcal{H}_x}^{p-2}v(x), \ \ v\in Im\:\partial_p.\]
We have $\left\|a(v)\right\|_{L_q(X,m,(\mathcal{H}_x)_x)}=\left\|v\right\|_{L_p(X,m,(\mathcal{H}_x)_x)}^{p-1}$ by H\"older's inequality, hence (\ref{E:growthp}) holds. Condition (\ref{E:hemicont}) rewrites
\begin{multline}
\lim_{\lambda\to 0}\int_X\left(\left\|\partial (u+\lambda v)(x)\right\|_{\mathcal{H}_x}^{p-2}\left\langle\partial(u+\lambda v)(x),\partial w(x)\right\rangle_{\mathcal{H}_x}\right.\notag\\
-\left. \left\|\partial u(x)\right\|_{\mathcal{H}_x}^{p-2}\left\langle\partial u(x),\partial w(x)\right\rangle_{\mathcal{H}_x}\right)m(dx)=0.
\end{multline}
But this follows by dominated convergence, the pointwise limit being obvious from the continuity of $\left\|\cdot\right\|_{\mathcal{H}_x}$ for fixed $x$ and a dominating integrable function being provided by
\[c\left(\left\| \partial v(x)\right\|_{\mathcal{H}_x}^{p-1}+\left\|\partial v(x)\right\|_{\mathcal{H}_x}^{p-1}\right)\left\|\partial w(x)\right\|_{\mathcal{H}_x}^{p-1}\ .\]
(\ref{E:monotonep}) holds with $c_4=0$ because 
\begin{multline}
\int_X\left(\left\|\partial f(x)\right\|_{\mathcal{H}_x}^p+\left\|\partial g(x)\right\|_{\mathcal{H}_x}^p-\left\|\partial f(x)\right\|_{\mathcal{H}_x}^{p-2}\left\langle \partial f(x),\partial g(x)\right\rangle_{\mathcal{H}_x}\right.\notag\\
\left. -\left\|\partial g(x)\right\|_{\mathcal{H}_x}^{p-2}\left\langle \partial f(x),\partial g(x)\right\rangle_{\mathcal{H}_x}\right)m(dx)\notag\\
\geq \int_X\left(\left\|\partial f(x)\right\|_{\mathcal{H}_x}^ {p-1}-\left\|\partial g(x)\right\|_{\mathcal{H}_x}^ {p-1}\right)\left(\left\|\partial f(x)\right\|_{\mathcal{H}_x}-\left\|\partial g(x)\right\|_{\mathcal{H}_x}\right)
\geq 0.
\end{multline}
Condition (\ref{E:coercivep}) follows immediately if a \emph{$p$-Poincar\'e inequality} is satisfied,
\begin{equation}\label{E:pPoincare}
\left\|f\right\|_{L_p(X,m)}^p\leq c_P\int_X\left\|\partial f(x)\right\|_{\mathcal{H}_x}^ pm(dx)
\end{equation}
with some $c_P>0$ for all $f\in L_p(X,m)$ with $\int_X fdm=0$. For smooth bounded Euclidean domains (\ref{E:pPoincare}) follows by classical arguments. It also holds if $(\mathcal{E},\mathcal{F})$ is induced by a regular resistance form \cite{Ki01, Ki03, Ki12} and is subject to Dirichlet boundary conditions. 
\end{examples}

\section{Stochastic calculus}\label{S:stochcalc}

In this final section we comment on the natural connection between the approach to $1$-forms by Cipriani and Sauvageot \cite{CS03, CS09} and the theory of stochastic integrals for continuous additive functionals as introduced by Nakao \cite{N85} and further investigated in \cite{ChFKuZh06, ChFKuZh08, FK04, LLZh98, LZh94}. Although this connection is known to experts (see for instance \cite[p. 506]{LZh94} or \cite[Example 5.6.1]{FOT94}), we would like to bring it to the attention of a broader audience. The setup studied in Sections \ref{S:Energy} and \ref{S:gradient} can be translated into probability and in particular, the notions of gradient and divergence have probabilistic counterparts. We finally point out that under mild conditions known perturbation results for symmetric Markov processes go well together with our fiberwise perspective on $\mathcal{H}$ and  lead to some analogs of classical non-divergence form operators.

We assume that $\mu$ is an admissible reference measure and $(\mathcal{E},\mathcal{F})$ is a symmetric regular Dirichlet form on $L_2(X,\mu)$. Let $Y=(Y_t)_{t\geq 0}$ denote the \emph{$\mu$-symmetric Hunt process} on $X$ uniquely associated with $(\mathcal{E},\mathcal{F})$ satisfying (\ref{E:desirable}). For background, notation and some important subtle details we refer to \cite{FOT94}. We consider \emph{additive functionals (AF's)} of $Y$, see for instance \cite[Section 5]{FOT94}. Given an AF $A=(A_t)_{t\geq 0}$ of $Y$, its \emph{energy} is defined by
\[\mathbf{e}(A):=\lim_{t\to 0}\frac{1}{2t}\mathbb{E}_\mu(A_t^2).\]
Let
\begin{multline}
\mathring{\mathcal{M}}:=\left\lbrace M: \text{$M$ is a finite cadlag AF of $Y$ with $\mathbf{e}(M)<\infty$ such that}\right.\notag\\
\left. \text{ for any $t>0$ we have $\mathbb{E}_x(M_t^2)<\infty$ and $\mathbb{E}_x(M_t)=0$ for q.e. $x\in X$}\right\rbrace,
\end{multline}
usually referred to as the \emph{space of martingale additive functionals of finite energy}. By polarization the energy $\mathbf{e}$ provides a bilinear form on $\mathring{\mathcal{M}}$ such that $(\mathring{\mathcal{M}},\mathbf{e})$ is Hilbert. Given $M\in\mathring{\mathcal{M}}$ let $\mu_{<M>}$ denote its \emph{energy measure} (the \emph{Revuz measure} of its \emph{sharp bracket} $<M>$) and for $M,N\in\mathring{\mathcal{M}}$ write $\mu_{<M,N>}$ for its bilinear version.
For any $M\in\mathring{\mathcal{M}}$ and $f\in L_2(X,\mu_{<M>})$ the \emph{stochastic integral} $f\bullet M \in\mathring{\mathcal{M}}$ is defined by the identity
\begin{equation}\label{E:stochint}
\mathbf{e}(f\bullet M,N)=\frac{1}{2}\int_Xf(x)\mu_{<M,N>}(dx), \ \ N\in \mathring{\mathcal{M}},
\end{equation}
The map $f\mapsto f\bullet M$ from $L_2(X,\mu_{<M>})$ into $\mathring{\mathcal{M}}$ is linear and continuous. See \cite[Theorem 5.6.1]{FOT94}. Note that by the nesting property of smooth measures, \cite[Section 2.2]{FOT94}, any $f\in C_0(X)$ is in $L_2(X,\mu_{<M>})$ for any $M\in\mathring{\mathcal{M}}$. To an energy finite function $u\in\mathcal{F}$ we can associate an additive functional $A^{[u]}$ by
\[A_t^{[u]}:=\widetilde{u}(Y_t)-\widetilde{u}(Y_0), \ \ t>0,\]
where as before $\widetilde{u}$ denotes the quasi-continuous Borel version of $u$. According to Fukushima's theorem, \cite[Theorem 5.2.2]{FOT94}, the functional $A^{[u]}$ decomposes uniquely as
\[A^{[u]}=M^{[u]}+N^{[u]},\]
where $M^{[u]}\in \mathring{\mathcal{M}}$ and $N^{[u]}$ is a member of the space 
\begin{multline}
\mathcal{N}_c:=\left\lbrace N: \text{ $N$ is a finite continuous AF with $\mathbf{e}(N)=0$ and such that } \right.\notag\\
\left. \text{$\mathbb{E}_x(|N_t|)<\infty$ for q.e. $x\in X$ for each $t>0$.}\right\rbrace
\end{multline}
of \emph{continuous additive functionals of zero energy}. For $u,w\in\mathcal{F}$ we have 
\[\mu_{<M^{[u]}, M^{[w]}>}=\Gamma(u,w).\]
From (\ref{E:scalarprodH}) and (\ref{E:stochint}) it is easily seen that 
\begin{equation}\label{E:isotheta}
\Theta(f\partial u):=2f\bullet M^{[u]}
\end{equation}
defines a linear isometry from the subspace $\mathcal{C}\otimes\mathcal{C}$ of $\mathcal{H}$ into $\mathring{\mathcal{M}}$ satisfying
\[\mathbf{e}(\Theta(f\partial u))=\left\|f\partial u\right\|_\mathcal{H}^2.\]
(The factor $2$ in (\ref{E:isotheta}) could easily be avoided by modifying related definitions, but to keep things simple we stick to those used in the literature.)
Now recall that by Remark \ref{R:CC} the space $\mathcal{C}\otimes \mathcal{C}$ is dense in $\mathcal{H}$. Therefore $\Theta$ extends to a linear isometry from $\mathcal{H}$ onto its image. However, \cite[Lemma 5.6.3]{FOT94} implies that the family
\[\left\lbrace f\bullet M^{[u]}: f,u\in\mathcal{C}\right\rbrace\]
is dense in $\mathring{\mathcal{M}}$. By a totality argument the range of $\Theta$ must therefore be all of $\mathring{\mathcal{M}}$, and we have reproved the following theorem, which is due to Nakao, \cite[Theorem 5.1]{N85}.

\begin{theorem}
The spaces $\mathcal{H}$ and $\mathring{\mathcal{M}}$ are isometrically isomorphic under the map $\Theta$ determined by (\ref{E:isotheta}).
\end{theorem}

By arguments similar to those in the proof of Lemma \ref{L:nonnegmeasure} we obtain the following result on energy measures.
\begin{corollary}
For the energy measure of $M\in\mathring{\mathcal{M}}$ we have 
\[\mu_{<M>}=\Gamma_\mathcal{H}(\omega), \ \ \text{ where }\ \ \omega=\Theta^ {-1}(M).\]
\end{corollary}

We can also reinterpret the gradient $\partial u$ of an energy finite function $u\in\mathcal{F}$ as the element uniquely corresponding to 
the martingale part of $A^{[u]}$.
\begin{corollary}
The image of the space $Im\:\partial$ under $\Theta$ is the subspace $\left\lbrace M^{[u]}: u\in\mathcal{F}\right\rbrace$ of $\mathring{\mathcal{M}}$, and for any $u\in\mathcal{F}$ we have $M^{[u]}=\frac{1}{2}\Theta(\partial u)$.
\end{corollary}

Also the divergence $\partial^\ast v$ of a vector field $v\in\mathcal{H}$ has probabilistic meaning. We briefly recall a construction from \cite[Section 3]{N85}. To simplify notation let us assume that $(\mathcal{E},\mathcal{F})$ is conservative. Set
\[\lambda(h;M):=\frac12\mu_{<M^{h},M>}(X)\ \ \text{ for $h\in\mathcal{F}$ and $M\in\mathring{\mathcal{M}}$}.\]
For fixed $M\in\mathring{\mathcal{M}}$ define a continuous additive functional $\Gamma(M)$ by
\[\Gamma(M)_t=N^{[w]}_t-\int_0^tw(Y_s)ds, \]
where $w\in\mathcal{F}$ is the unique function such that 
\[\lambda(h;M)=\mathcal{E}_1(w,h).\]
The functional $\Gamma(M)$ is characterized by the validity of
\begin{equation}\label{E:charactident}
\lim_{t\to 0}\frac{1}{t} \mathbb{E}_{h\cdot \mu}[\Gamma(M)_t]=-\lambda(h;M)
\end{equation}
for any $h\in\mathcal{C}$. In \cite{N85} the functional $\Gamma(M)$ had been used to define It\^o and Stratonovich type integrals with respect to continuous additive functionals, see also \cite{ChFKuZh06, ChFKuZh08}. The following first observation is immediate.

\begin{lemma}
Let $M\in\mathring{\mathcal{M}}$ and $\omega=\Theta^{-1}(M)$. Then we have
\[\partial^\ast \omega(h)=-\lambda(h;M)\]
for all $h\in\mathcal{C}$.
\end{lemma}

Recall that to each continuous additive functional $A=(A_t)_{t\geq 0}$ there exists a unique signed smooth Borel measure $\mu_A$ such that 
\begin{equation}\label{E:Revuzcorr}
\lim_{t\to 0}\frac1t\mathbb{E}_{h\cdot \mu}((f\cdot A)_t)=\int_X \widetilde{h} fd\mu_A
\end{equation}
for any $f\in \mathcal{B}_b(X)$ and any nonnegative $h\in\mathcal{F}$, as follows from \cite[Theorem 5.1.4]{FOT94}. In this case $\mu_A$ is called the \emph{signed Revuz measure of $A$}. Given $M\in\mathring{\mathcal{M}}$, let $\mu_{\Gamma(M)}$ denote the signed Revuz measure of $\Gamma(M)$. 
From (\ref{E:charactident}) together with the previous lemma we obtain the following.

\begin{corollary}
Let $M\in\mathring{\mathcal{M}}$ and $\omega=\Theta^{-1}(M)$. Then the element $\partial^\ast\omega$ of $\mathcal{C}^\ast$
can be represented by integration with respect to $\mu_{\Gamma(M)}$, 
\[\partial^\ast\omega(h)=\int_Xhd\mu_{\Gamma(M)}, \ \ h\in\mathcal{C}.\] 
Moreover, if $\omega\in dom\:\partial^\ast$, then $\mu_{\Gamma(M)}$ is absoulutely continuous with respect to $\mu$, and its density is $\partial^ \ast\omega\in L_2(X,\mu)$.
\end{corollary} 
\begin{proof}
For nonnegative $h\in\mathcal{C}$ identity (\ref{E:charactident}) together with (\ref{E:Revuzcorr}) yields $\int_X h d\mu_{\Gamma(M)}=-\lambda(h;M)$, and the first statement follows by linearity and the previous lemma. The second statement is an immediate consequence.
\end{proof}

\begin{examples}
For $u\in\mathcal{C}$ and $M=2M^{[u]}=\Theta(\partial u)$ we have 
\[\partial^\ast\partial u(h)=-\mathcal{E}(u,h)=\lim_{t\to 0}\frac1t\mathbb{E}_{h\cdot \mu}[N^{[u]}_t],\]
and for $-Lu$, viewed as a member of $\mathcal{C}^\ast$ as in (\ref{E:goodcase}), we obtain the measure representation
\[(-Lu)(h)=\int_X fd\mu_{\Gamma(M)}.\]
If in addition $u\in dom\:L$,  then $\partial u\in dom\:\partial^\ast$ and we have $d\mu_{\Gamma(M)}=-Lu\:d\mu$, seen as an equality of signed measures.
\end{examples}
 
We finally consider a simple version of a related perturbation result. To do so, we assume that $m$ is a measure satisfying Assumption \ref{A:energydominant} and that $(\mathcal{E},\mathcal{C})$ is closable on $L_2(X,m)$ with closure $(\mathcal{E},\mathcal{F}^{(m)})$. Recall that this is particularly true for $m=\mu$ if $\mu$ itself satisfies Assumption \ref{A:energydominant}. Let 
\[\mathcal{H}_\infty:=\left\lbrace v\in\mathcal{H}: \Gamma_{\mathcal{H},\cdot}(v)\in L_\infty(X,m)\right\rbrace\]
denote the space of \emph{vector fields of bounded length}. Given $b\in\mathcal{H}_\infty$ and $u\in\mathcal{F}^{(m)}$, Cauchy-Schwarz yields
\[\int_X\widetilde{u}^2\Gamma_{\mathcal{H},\cdot}(b)dm\leq \left\|\Gamma_{\mathcal{H},\cdot}(b)\right\|_{L_\infty(X,m)}\left\|u\right\|_{L_2(X,m)}.\]
(Here and in the following the quasi-continuous versions are taken with respect to $(\mathcal{E},\mathcal{F}^{(m)})$.)
In particular, the measure $\Gamma_{\mathcal{H},\cdot}dm$ is of the \emph{Hardy class}, cf. \cite[p. 141]{FK04}. 

If $b,\hat{b}\in\mathcal{H}_\infty$ and $c\in L_\infty(X,m)$ then the bilinear form 
\begin{multline}\label{E:formQ}
\mathcal{Q}(f,g):=\mathcal{E}(f,g)-\int_X\widetilde{g}(x)\left\langle b_x,\partial_x f\right\rangle_{\mathcal{H}_x}m(dx)
-\int_X\widetilde{f}(x)\left\langle \hat{b}_x,\partial_x g\right\rangle_{\mathcal{H}_x}m(dx)\\
-\int_X f(x)g(x)c(x)m(dx),
\end{multline}
$f,g\in\mathcal{F}^{(m)}$, is closed on $L_2(X,m)$, as can be seen by standard perturbation arguments. Moreover, the strongly continuous $L_2$-semigroup associated with
\[\mathcal{Q}_\alpha(f,g):=\mathcal{Q}(f,g)+\alpha\left\langle f,g\right\rangle_{L_2(X,m)}\]
is \emph{positivity preserving} if $\alpha>0$ is large enough. See for instance \cite{FK04} (in particular p. 142) and the references therein. In general this semigroup may possibly fail to be Markovian, cf. \cite{LLZh98}. 

If $\hat{b}\in dom\:\partial^\ast$, then we have
\begin{align}
\partial^\ast(u\hat{b})(h)&=-\left\langle \hat{b},u\partial h\right\rangle_\mathcal{H}\notag\\
&=-\left\langle \hat{b},\partial (uh)\right\rangle_\mathcal{H}+\left\langle \hat{b},h\partial u\right\rangle_\mathcal{H}\notag\\
&=\int_X h(x)(\partial^\ast\hat{b})(x)u(x)m(dx)+\int_X h(x)\left\langle \hat{b}_x,\partial_x u\right\rangle_{\mathcal{H}_x}m(dx)\notag
\end{align}
for any $h,u\in\mathcal{C}$. That is, we can reinterpret $\partial^\ast(u\hat{b})$ as a measure that is absolutely continuous with respect to $m$ and has a density given by
\[\partial^\ast(u\hat{b})(x)=(\partial^\ast \hat{b})(x)u(x)+\left\langle b_x\partial_xu\right\rangle_{\mathcal{H}_x} \]
for $m$-a.a. $x\in X$. For the generator $L^\mathcal{Q}$ of $\mathcal{Q}$ and any function $u\in \mathcal{C}\cap dom\:L^{(m)}$ we therefore have
\[L^\mathcal{Q}u(x)=L^{(m)}u(x)+\left\langle b_x,\partial_xu\right\rangle_{\mathcal{H}_x}-\partial^\ast(u\hat{b})(x)+c(x)u(x)\]
for $m$-a.a. $x\in X$. This is a generalization of formula (1.12) in \cite[Example 1.1]{FK04}. The operator $L^{\mathcal{Q}}$ may be seen as the analog of a non-divergence form operator in our context.


\begin{thebibliography}
\normalsize
\bibitem {Ba} M. T. Barlow,
\emph{Diffusions on fractals.}\
 Lectures on Probability Theory and Statistics (Saint-Flour, 1995), 1--121,
Lecture Notes in Math., \textbf{1690},
Springer, Berlin, 1998.

\bibitem{BB99}
M.T. Barlow, R.F. Bass, \emph{Brownian motion and harmonic analysis on Sierpinski carpets}, Canad. J. Math. 51 (1999), 673--744.

\bibitem{BBKT} M. T. Barlow, R. F. Bass, T. Kumagai, and A. Teplyaev, \emph{Uniqueness of Brownian motion on Sierpinski carpets.}\ 
J. Eur. Math. Soc. 
{\bf12} (2010), 655--701. 

\bibitem{BdPR}
V. Barbu, G. DaPrato, M. R\"ockner, \emph{Stochastic porous media equation and self-organized criticality},
Comm. Math. Phys. 285 (2009) no.3, 901-923.

\bibitem{BBST99}
O. Ben-Bassat, R.S. Strichartz, A. Teplyaev, \emph{What is not in the domain of Sierpinski gasket type fractals},
J. Funct. Anal. 166 (1999), 197--217.

\bibitem{BH91}
N. Bouleau, F. Hirsch, \emph{Dirichlet Forms and Analysis on Wiener Space},
deGruyter Studies in Math. 14, deGruyter, Berlin, 1991.

\bibitem{BM91}
M. Biroli, U. Mosco, \emph{Formes de Dirichlet et estimation structurelles dans les milieux discontinus},
C.R. Acad. Sci. Paris 313 (1991), 593-598.

\bibitem{BM95}
M. Biroli, U. Mosco, \emph{A Saint-Venant principle for Dirichlet forms on discontinuous media},
Ann. Mat. Pura Appl. 169 (4) (1995), 125-181.

\bibitem{ChFKuZh06}
Z.-Q. Chen, P.J. Fitzsimmons, K. Kuwae, T.-S. Zhang, \emph{Stochastic calculus for symmetric Markov processes},
Ann. Probab. {\bf 36} (3) (2008), 931-970. 

\bibitem{ChFKuZh08}
Z.-Q. Chen, P.J. Fitzsimmons, K. Kuwae, T.-S. Zhang, \emph{Perturbations of symmetric Markov processes},
Probab. Th. Relat. Fields {\bf 140} (2008), 239-275.

\bibitem{ChFu12}
Z.-Q. Chen, M. Fukushima, \emph{Symmetric Markov Processes, Time Change, and Boundary Theory},
Princeton Univ. Press, Princeton, 2012. 

\bibitem{CS03}
F. Cipriani, J.-L. Sauvageot, \emph{Derivations as square roots of Dirichlet forms},
J. Funct. Anal. 201 (2003), 78-120.

\bibitem{CS09}
F. Cipriani, J.-L. Sauvageot, \emph{Fredholm modules on p.c.f. self-similar fractals and their conformal geometry},
Comm. Math. Phys. 286 (2009), 541-558.

\bibitem{D89}
E.B. Davies, \emph{Heat kernels and spectral theory}, 
Cambridge Univ. Press, Cambridge, 1989.

\bibitem{Dix}
J. Dixmier, \emph{Von Neumann Algebras},
North-Holland Math. Lib. 27, North-Holland, Amsterdam, 1981.

\bibitem{Eb99}
A. Eberle, \emph{Uniqueness and non-uniqueness of semigroups generated by singular diffusion operators}, Springer LNM 1718, Springer, New York, 1999.

\bibitem{Ev98}
L.C. Evans, \emph{Partial Differential Equations},
Grad. Stud. Math. vol 19, AMS, Providence, RI, 1998.

\bibitem{FaHu01}
K.J. Falconer, J. Hu, \emph{Nonlinear Diffusion Equations on unbounded fractal domains},
J. Math. Anal. Appl. 256 (2001), 606-624.

\bibitem{FK04}
P.J. Fitzsimmons, K. Kuwae, \emph{Non-symmetric perturbations of symmetric Dirichlet forms},
J. Funct. Anal. {\bf 208} (2004), 140-162.

\bibitem{FLJ}
M. Fukushima, Y.LeJan, \emph{On quasi-supports of smooth measures and closability of pre-Dirichlet forms},
Osaka J. Math. {\bf 28} (1991), 837-845.

\bibitem{FOT94}
M. Fukushima, Y. Oshima and M. Takeda, \textit{Dirichlet forms and symmetric Markov processes},
deGruyter, Berlin, New York, 1994.


\bibitem{FST}
M. Fukushima, K. Sato, S. Taniguchi, \emph{On the closable parts of pre-Dirichlet forms and the fine supports of underlying measures},
Osaka J. Math. {\bf 28} (1991), 517-535.



\bibitem{HMT} B. M. Hambly, V. Metz and A. Teplyaev, \emph{Self-similar energies on post-critically finite self-similar fractals.} J. London
Math. Soc. {\bf74} (2006), 93--112.

\bibitem{HPS04}
P.E. Herman, R. Peirone, R.S. Strichartz, \emph{$p$-energy and $p$-harmonic functions on Sierpinski gasket type fractals},
Pot. Anal. 20 (2004), 125-148.

\bibitem{Hino03}
M. Hino, \emph{On singularity of energy measures on self-similar sets},
Probab. Theory Relat. Fields 132 (2005), 265-290.

\bibitem{Hino05}
M. Hino, K. Nakahara, \emph{On singularity of energy measures on self-similar sets II},
Bull. London Math. Soc. {\bf38} (2006), 1019-1032.

\bibitem{Hino08}
M. Hino, \emph{Martingale dimensions for fractals.},
 Ann. Probab. {\bf36} (2008),   971--991.

\bibitem{Hino10}
M. Hino, \emph{Energy measures and indices of Dirichlet forms, with applications to derivatives on some fractals},
Proc. London Math. Soc. {\bf100}  (2010), 269-302.

\bibitem{H12}
M. Hinz, \emph{$1$-forms and polar decomposition on harmonic spaces}.
 Potential Anal. {\bf38} (2013),   261--279.

\bibitem{HTa} 
M. Hinz, A. Teplyaev, \emph{Local Dirichlet forms, Hodge theory, and the 
Navier-Stokes equations on topologically one-dimensional  fractals}. 
to appear in the Transactions of the American Mathematical Society, 
 arXiv:1206.6644

\bibitem{HTb}
M. Hinz, A. Teplyaev, \emph{Dirac and magnetic Schr\"odinger operators on fractals},
preprint (2012) arXiv:1207.3077

\bibitem{HTc} 
M. Hinz, A. Teplyaev, \emph{Vector analysis on fractals and applications},
to appear in 
Contemporary Math. 
arXiv:1207.6375 

\bibitem{HKT}
M. Hinz,  D. Kelleher,   A. Teplyaev, \emph{Measures and Dirichlet forms under the Gelfand transform}, 
Probability and statistics {\bf18}, 
Zapiski Nauchnyh Seminarov POMI {\bf408}, (2012), 303--322;
reprinted in  
\href{http://www.springer.com/mathematics/journal/10958}{Journal of Math. Sciences}   (Springer, 2013) arXiv:1212.1099

\bibitem{HKT2013}
M. Hinz,  D. Kelleher,   A. Teplyaev, \emph{Metrics and spectral triples for Dirichlet and resistance forms},   
preprint  



\bibitem{IRT}
M. Ionescu, L. Rogers, A. Teplyaev, \emph{Derivations, Dirichlet forms and spectral analysis},
Journal of Functional Analysis
{\bf263}  (2012),  2141--2169.

\bibitem{Ki93h} J. Kigami,
\emph{Harmonic metric and Dirichlet form on the Sierpi\'nski gasket.}\
 Asymptotic
problems in probability theory: stochastic models and
diffusions on fractals (Sanda/Kyoto, 1990), 201--218,
Pitman Res. Notes Math. Ser., \textbf{283},
Longman Sci. Tech., Harlow, 1993.

\bibitem{Ki01}
J. Kigami, \emph{Analysis on Fractals}, Cambridge Univ. Press, Cambridge, 2001.

\bibitem{Ki03}
J. Kigami, \emph{Harmonic analysis for resistance forms}, J. Funct. Anal. 204 (2003), 525--544.

\bibitem{Ki08}
J. Kigami, \emph{Measurable Riemannian geometry on the Sierpinski gasket: the Kusuoka measure and the Gaussian heat kernel estimate}, Math. Ann. 340 (2008), 781-804.

\bibitem{Ki12}
J. Kigami, \emph{Resistance forms, quasisymmetric maps and heat kernel estimates}.  Mem. Amer. Math. Soc. 216 (2012), no. 1015

\bibitem{KR}
N.V. Krylov, B.L. Rozovskij, \emph{Stochastic evolution equations},
Journal Soviet Math. 16(4) (1981), 1233-1277.


\bibitem{Ku89}
S. Kusuoka, \emph{Dirichlet forms on fractals and products of random matrices},
Publ. Res. Inst. Math. Sci. 25 (1989), 659-680.

\bibitem{KN91}
K. Kuwae, Sh. Nakao, \emph{Time changes in Dirichlet space theory},
Osaka J. Math. {\bf 28}, 847-865.

\bibitem{LeJan78}
Y. LeJan, \emph{Mesures associ\'ees \`a une forme de Dirichlet. Applications.},
Bull. Soc. Math. France 106 (1978), 61-112.
\bibitem{Li} T. Lindstr{\o}m,
\emph{Brownian motion on nested fractals.}\
Mem.  Amer.  Math.  Soc. \textbf{420}, 1989.

\bibitem{LLZh98}
J. Lunt, T.J. Lyons, T.-S. Zhang, \emph{Integrability of functionals of Dirichlet processes, probabilistic representations of semigroups, and estimates of heat kernels},
J. Funct. Anal. {\bf 153} (1998), 320-342.

\bibitem{LZh94}
T.J. Lyons, T.-S. Zhang, \emph{Decomposition of Dirichlet processes},
Ann. Probab {\bf 22} (1) (1994), 494-524.

\bibitem{M80}
F.-Y. Maeda, \emph{Dirichlet Integrals on Harmonic Spaces}, Lecture Notes in Math. 803, Springer, New York, 1980.

\bibitem{N85}
Sh. Nakao, \emph{Stochastic calculus for continuous additive functionals},
Z. Wahrsch. verw. Geb. {\bf 68} (1985), 557-578.

\bibitem{Par}
E. Pardoux, \emph{\'Equations aux d\'eriv\'ees partielles stochastiques de type monotone},
S\'eminaire sur les \'Equations aux D\'eriv\'ees Partielles (1974-1975), III, Exp. no. 2, Coll\`ege de France, Paris, 1975, p.10.

\bibitem{P}R. Peirone, \emph{Existence of eigenforms on nicely separated fractals.}, 
Analysis on graphs and its applications, 231--241, Proc. Sympos. Pure Math., {\bf77}, Amer. Math. Soc., Providence, RI, 2008.
 \emph{Existence of Self-similar Energies on Finitely Ramified Fractals}, 
preprint, 2011.


\bibitem{PrevRoeck}
C. Pr\'ev\^ot, M. R\"ockner, \emph{A Concise Course on Stochastic Partial Differential Equations},
Springer LNM 1905, Springer, Berlin, 2007. 

\bibitem{RS}
M. Reed, B. Simon, \emph{Methods of Modern Mathematical Physics, vol. 1}, Acad. Press, San Diego 1980.

\bibitem{RRW}
J. Ren, M. R\"ockner, F.-Y. Wang, \emph{Stochastic generalized porous media and fast diffusion equations},
J. Diff. Equations {\bf238} (2007) no.1, 118-152.

\bibitem{RW}
M. R\"ockner and N. Wielens, \emph{Dirichlet forms---closability and change of speed measure},
{Infinite-dimensional analysis and stochastic processes
              ({B}ielefeld, 1983)},
 {Res. Notes in Math.},
{\bf124},
 {119--144},
 {Pitman},
{Boston, MA},
{1985}.


\bibitem{Stei10}
 
B. Steinhurst, \emph{Uniqueness of locally symmetric Brownian motion on Laakso spaces}. Potential Anal. {\bf38} (2013),   281--298.

\bibitem{Stoll10}
P. Stollmann, \emph{A dual characterization of length spaces with application to Dirichlet forms},
Studia Math. 198 (3) (2010), 221-233. 

\bibitem{Str06}
R.S. Strichartz, \emph{Differential Equations on Fractals: A Tutorial}, Princeton Univ. Press, Princeton 2006.

\bibitem{StrW04}
R.S. Strichartz, C. Wong, \emph{The $p$-Laplacian on the Sierpinski gasket}, Nonlinearity 17 (2004), 595-616.

\bibitem{Sturm94}
K.-Th. Sturm, \emph{Analysis on local Dirichlet spaces - I. Recurrence, conservativeness and $L^p$-Liouville properies},
J. reine angew. Math. 456 (1994), 173-196. 

\bibitem{Sturm95}
K.-Th. Sturm, \emph{On the geometry defind by Dirichlet forms},
In: Progress in Probability, Vol. 36, Birkh\"auser, Basel 1995, p. 231-242.

\bibitem{Tak}
M. Takesaki, \emph{Theory of Operator Algebras I},
Encycl. Math. Sci. 124, Springer, New York, 2002.

\bibitem{T08}
A. Teplyaev, \emph{Harmonic coordinates on fractals with finitely ramified cell structure}, Canad. J. Math. 60 (2008), 457--480.
\end{thebibliography}
\end{document}